\documentclass[11pt]{amsart}

\usepackage{amsmath, amsthm, amsfonts, amssymb,  mathrsfs, graphicx, lscape}
\usepackage{enumitem}

\usepackage[all]{xy}
\usepackage[usenames, dvipsnames]{color}
\usepackage[margin=1in]{geometry} 
\usepackage{tikz}
\usetikzlibrary{matrix,arrows}
\usepackage{tikz-cd}

\numberwithin{equation}{section}
\newtheorem{theorem}[equation]{Theorem}

\newtheorem{proposition}[equation]{Proposition}
\newtheorem{lemma}[equation]{Lemma}
\newtheorem{corollary}[equation]{Corollary}

\theoremstyle{definition}
\newtheorem{rmk}[equation]{Remark}
\newenvironment{remark}[1][]{\begin{rmk}[#1] \pushQED{\qed}}{\popQED \end{rmk}}
\newtheorem{eg}[equation]{Example}
\newenvironment{example}[1][]{\begin{eg}[#1] \pushQED{\qed}}{\popQED \end{eg}}
\newtheorem{defn}[equation]{Definition}
\newenvironment{definition}[1][]{\begin{defn}[#1]\pushQED{\qed}}{\popQED \end{defn}}
\newtheorem{ques}[equation]{Question}

\usepackage[in]{fullpage}
\usepackage[colorlinks=true, pdfstartview=FitV, linkcolor=black, citecolor=black, urlcolor=black]{hyperref}

\newcommand{\cK}{\mathcal{K}}

\newcommand{\cM}{\mathcal{M}}

\newcommand{\cO}{\mathcal{O}}

\newcommand{\cP}{\mathcal{P}}

\newcommand{\cQ}{\mathcal{Q}}

\newcommand{\bR}{\mathbf{R}}
\newcommand{\cR}{\mathcal{R}}

\newcommand{\cT}{\mathcal{T}}

\newcommand{\cY}{\mathcal{Y}}

\newcommand{\cZ}{\mathcal{Z}}

\newcommand{\bd}{\mathbf{d}}

\newcommand{\be}{\mathbf{e}}

\newcommand{\br}{\mathbf{r}}

\newcommand{\bs}{\mathbf{s}}

\newcommand{\za}{\ensuremath{\alpha}}
\newcommand{\zb}{\ensuremath{\beta}}

\newcommand{\zl}{\ensuremath{\lambda}}

\newcommand{\ZZ}{\mathbb{Z}}

\newcommand{\bbA}{\mathbb{A}}
\newcommand{\bbD}{\mathbb{D}}
\newcommand{\SSS}{\mathbb{S}}

\renewcommand{\phi}{\varphi}
\renewcommand{\emptyset}{\varnothing}

\renewcommand{\tilde}[1]{\widetilde{#1}}
\newcommand{\ol}[1]{\overline{#1}}

\newcommand{\setst}[2]{\left\{ #1 \mid #2 \right\}}

\newcommand{\comment}[1]{}

\makeatletter
\def\Ddots{\mathinner{\mkern1mu\raise\p@
\vbox{\kern7\p@\hbox{.}}\mkern2mu
\raise4\p@\hbox{.}\mkern2mu\raise7\p@\hbox{.}\mkern1mu}}
\makeatother

\DeclareMathOperator{\im}{image}

\DeclareMathOperator{\rad}{rad}

\DeclareMathOperator{\rank}{rank}

\DeclareMathOperator{\End}{End}

\DeclareMathOperator{\Sym}{Sym}

\DeclareMathOperator{\gr}{gr}

\DeclareMathOperator{\module}{mod}
\DeclareMathOperator{\rep}{rep}
\DeclareMathOperator{\Mat}{Mat}

\DeclareMathOperator{\Gr}{Grass}

\newcommand{\op}{\operatorname}
\newcommand{\oo}{\otimes}

\newcommand{\onto}{\twoheadrightarrow}
\newcommand{\into}{\hookrightarrow}

\usepackage{stmaryrd}

\usepackage[normalem]{ulem}

\DeclareMathOperator{\charac}{char}
\DeclareMathOperator{\soc}{soc}
\newsymbol\pp 1275
\newcommand{\kk}{\Bbbk}

\makeatletter
\@namedef{subjclassname@2020}{%
  \textup{2020} Mathematics Subject Classification}
\makeatother

\begin{document}
\title{Representation varieties of algebras with nodes}
\author{Ryan Kinser}
\address{University of Iowa, Department of Mathematics, Iowa City, IA, USA}
\email[Ryan Kinser]{ryan-kinser@uiowa.edu}

\author{Andr\'as C. L\H{o}rincz}
\address{Humboldt--Universit\"at zu Berlin, Institut f\"ur Mathematik, Berlin, Germany}
\email[Andr\'as C. L\H{o}rincz]{lorincza@hu-berlin.de}

\begin{abstract}
We study the behavior of representation varieties of quivers with relations under the operation of node splitting. We show how splitting a node gives a correspondence between certain closed subvarieties of representation varieties for different algebras, which preserves properties like normality or having rational singularities. 
Furthermore, we describe how the defining equations of such closed subvarieties change under the correspondence.

By working in the ``relative setting'' (splitting one node at a time), we demonstrate that there are many non-hereditary algebras whose irreducible components of representation varieties are all normal with rational singularities.
We also obtain explicit generators of the prime defining ideals of these irreducible components. 
This class contains all radical square zero algebras, but also many others, as illustrated by examples throughout the paper.
We also show the above is true when replacing irreducible components by orbit closures, for a more restrictive class of algebras.
Lastly, we provide applications to decompositions of moduli spaces of semistable representations of certain algebras.
\end{abstract}



\maketitle



\section{Introduction}
\subsection{Context and motivation}\label{sec:context}
Throughout, $\kk$ is an algebraically closed field.  We explicitly mention when it is necessary to specialize to characteristic 0.
The algebras we study are those of the form $A=\kk Q/I$ where $Q$ is a quiver and $I$ a two-sided ideal. 

Each dimension vector $\bd$ for $A$ determines a representation variety $\rep_A(\bd)$ with action of a product of general linear groups $GL(\bd)$ (see Section \ref{sec:repvariety}).
Orbit closures in $\rep_A(\bd)$ have remarkable connections with the representation theory of $A$ and related objects; 
see surveys such as \cite{Bongartzsurvey,Zwarasurvey,HZsurvey} for detailed treatments.
Interest in these varieties is not confined to representation theory of algebras, however: they also naturally arise in Lie theory, commutative algebra, and algebraic geometry.  The interested reader may consult the introduction to \cite{KinserICRA} for more detail and references.

Even restricting our attention to representation theory of algebras, geometric methods centered around the varieties $\rep_A(\bd)$, such as the construction of moduli spaces (see Section \ref{sec:moduli}), provide a toolkit for classification of representations which is complementary to homological and functorial approaches.

The purpose of this paper is to relate representation varieties for algebras related by splitting nodes, which we  informally recall here. Node splitting is a well-known technique which was introduced by Mart\'inez-Villa \cite{MV80}; see Section \ref{sec:nodesplit} for details.
In this paper, a \emph{node} of an algebra $A=\kk Q/I$ is a vertex $x$ of $Q$ such that all the paths of length $2$ passing strictly through $x$ belong to $I$. (The traditional definition of ``node'' further requires that $x$ be neither a sink nor source.)
A node $x$ of $A$ can be \emph{split} by the following local operation around $x$:
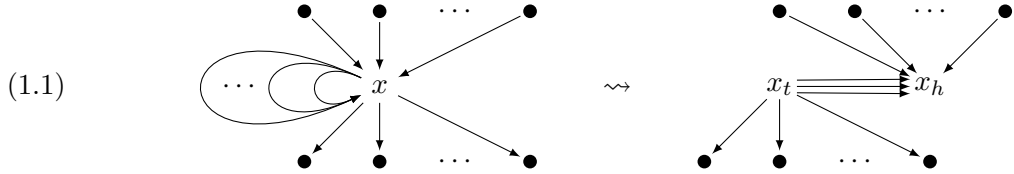
\begin{equation}
\vcenter{\hbox{\begin{tikzpicture}[point/.style={shape=circle,fill=black,scale=.5pt,outer sep=3pt},>=latex]
   \node[point] (y1) at (-2,1) {};
  \node[point] (y2) at (-1,1) {};
  \node[outer sep=-2pt] at (0,1) {${\cdots}$};
  \node[point] (yr) at (1,1) {};
   \node[label={[label distance=0.5cm]180:$\cdots$}] (x) at (-1,0) {${x}$};
   \node[point] (z1) at (-2,-1) {};
  \node[point] (z2) at (-1,-1) {};
  \node[outer sep=-2pt] at (0,-1) {${\cdots}$};
  \node[point] (zs) at (1,-1) {};
\draw[->] (x) to [out=155,in=205,looseness=10] (x);
\draw[->] (x) to [out=150,in=210,looseness=20] (x);
  \path[->]
	(y1) edge  (x)
	(y2) edge  (x)
	(yr) edge  (x)
  	(x) edge  (z1)
	(x) edge  (z2)
	(x) edge  (zs);
   \end{tikzpicture}}}
\qquad\rightsquigarrow \qquad
\vcenter{\hbox{\begin{tikzpicture}[point/.style={shape=circle,fill=black,scale=.5pt,outer sep=3pt},>=latex]
   \node[point] (y1) at (-1,1) {};
  \node[point] (y2) at (0,1) {};
  \node[outer sep=-2pt] at (1,1) {${\cdots}$};
  \node[point] (yr) at (2,1) {};
   \node[outer sep=-2pt] (xt) at (-1,0) {${x_t}$};
   \node[outer sep=-2pt] (xh) at (1,0) {${x_h}$};
   \node[point] (z1) at (-2,-1) {};
  \node[point] (z2) at (-1,-1) {};
  \node[outer sep=-2pt] at (0,-1) {${\cdots}$};
  \node[point] (zs) at (1,-1) {};
  \node at (0,.1) {$\vdots$};
  \path[->]
	(y1) edge  (xh)
	(y2) edge  (xh)
	(yr) edge  (xh)
  	(xt) edge  (z1)
	(xt) edge  (z2)
	(xt) edge  (zs);
\draw[->,bend left=20] (xt.20) to (xh.160);
\draw[->,bend right=20] (xt.-20) to (xh.200);
   \end{tikzpicture}}}
\end{equation}
resulting in a new algebra $A^x$ with one fewer node (disregarding sources and sinks). 
Algebraically, it is easy to see that the category of left $A$-modules is equivalent to the full subcategory of left $A^x$-modules whose objects have no direct summand isomorphic to the simple at $x_h$.
Geometrically, however, the representation varieties can be quite different, as the bijection on isomorphism classes translates to a bijection of orbits, but not an isomorphism of varieties (even after stratification). See Example \ref{ex:stratfail} for more detail.

The technique of this paper is to relate the representation varieties of $A$ and $A^x$ by a homogeneous fiber bundle construction and ``collapsing maps'' in the sense of Kempf \cite{Kempf76}.  This allows us to relate their singularities, in particular whether they are normal or have rational singularities, and relate defining equations of the prime ideals of equivariant closed subvarieties as well.

A special case of particular interest is when every vertex of $A$ is a node, which is precisely when $\rad^2 A = 0$.  Such algebras are historically significant because they are the first step beyond semisimple algebras.  Interesting remarks on the importance of radical square zero algebras and their associated graphs in the development of the modern representation theory of algebras can be found in the volume of Gabriel and Roiter \cite[\S\S 7.8, 8.7]{GREMS}.
Turning to the history of geometry of representations of algebras, one can consider Buchsbaum-Eisenbud varieties of complexes as representation varieties of the radical square zero algebra associated to the quiver $\bullet \to \bullet \to \cdots \to \bullet$.  These varieties were studied extensively in the 1970s and results on them were eventually generalized beyond the radical square zero case (for this particular quiver) \cite{ADK81,LM98}.  See the introduction of \cite{KR} for more remarks and references on this.  So from both representation theoretic and geometric perspectives, we can see the radical square zero case as an important starting point for deeper developments.

While radical square zero algebras are an important special case, we emphasize that our results apply more generally; see Examples \ref{ex:irrcomp}, \ref{ex:gendecomp}, \ref{ex:singularity}, \ref{ex:singularity3}, \ref{ex:rad2moduli}, and \ref{ex:moduli}.

\subsection{Statement of main results}\label{sec:results}
Our main results require that $\charac\kk=0$, aside from one exception noted after Corollary \ref{cor:radsquarezero}, so we make this assumption for this subsection.
Let $A=\kk Q/I$ be an algebra with node $x$.
We start by giving a precise geometric description of how representation varieties of $A$ and $A^x$ are related in Section \ref{sec:nodebundles}. 

Given an irreducible $GL(\bd)$-equivariant closed subvariety $C \subset \rep_A(\bd)$ (e.g. an irreducible component or orbit closure), Proposition \ref{prop:bijection} uniquely associates to it an irreducible $GL(\be)$-equivariant closed subvariety $C^x \subset \rep_{A^x}(\be)$
for a dimension vector $\be$ of $A^x$.
Typically, $C^x$ is geometrically easier to understand than $C$ because $A^x$ is defined by fewer relations on a path algebra than $A$.   We make this precise in the following two theorems.
Our first main result is the following theorem on singularities.

\begin{theorem}\label{thm:mainsing}
The variety $C^x$ is normal if and only if the variety $C$ is normal. Moreover, $C^x$ has rational singularities if and only if $C$ has rational singularities.
\end{theorem}

Next, we describe how the prime defining ideals of irreducible components change under node splitting (for a more general statement, see Theorem \ref{thm:defideal}). For an arrow $\za \in Q_1$, we denote by $X_{\za}$ the $\bd(t\za)\times \bd(h\za)$ generic matrix of variables. We can identify the coordinate ring $\kk[\rep_{\kk Q}(\bd)]$ with a polynomial ring in the entries of the matrices $\{X_\za\}_{\za\in Q_1}$. For $x \in Q_0$, we write $H_x$ (resp. $T_x$) for the $\bd(x)\times \left(\displaystyle\sum_{h\za = x} \bd(t\za)\right)$ matrix (resp. $\left(\displaystyle\sum_{t\za = x} \bd(h\za)\right)\times \bd(x)$ matrix) obtained by placing all matrices $X_\za$ with $h\za = x$ next to (resp. with $t\za = x$ on top of) each other, in any order. 

\begin{theorem}\label{thm:maindef}
Let $\cP$ be a set of polynomials in $\kk[\rep_{\kk Q^x}(\be)]$ that generate the prime ideal defining $C^x$, and $r$ the rank of $H_x$ for a generic element in $C^x$. Then the prime ideal defining $C$ is generated by the following polynomials in $\kk[\rep_{\kk Q}(\bd)]$:
\begin{enumerate}
\item the $(r+1)\times (r+1)$ minors of $H_x$;
\item the $(\bd(x)-r+1) \times (\bd(x)-r+1)$ minors of $T_x$;
\item the entries of \, $T_x \cdot H_x$;
\item the trace of \,$X_{\gamma}$, for every loop $\gamma\in Q_1$ at $x$;
\item the $GL(\bd(x))$-saturation of the equations in $\cP$.
\end{enumerate}
\end{theorem}

We single out the case of radical square zero algebras for special attention.  In this case every vertex is a node, so splitting nodes results in a quiver without relations (i.e. a hereditary algebra), whose representation varieties are simply affine spaces.  This immediately yields the following result.

\begin{corollary}\label{cor:radsquarezero}
Let $A$ be a finite-dimensional $\kk$-algebra with $\rad^2 A= 0$. Then for any dimension vector $\bd$, each irreducible component $C \subset \rep_A (\bd)$ has rational singularities (and is thus also normal, and Cohen--Macaulay).
\end{corollary}

In fact, the claim about normality and the Cohen--Macaulay property holds in positive characteristic as well, see Remark \ref{rem:char}. To our knowledge, this is the first result limiting the singularities of all irreducible components of representation varieties for such a large class of non-hereditary, finite-dimensional algebras.
Similarly, Theorem \ref{thm:maindef} allows us to give an explicit set of defining equations for the irreducible components of representation varieties of radical square zero algebras (see Corollary \ref{cor:radef}).

\subsection{Relation to existing literature}\label{sec:lit}
When considering all algebras $A=\kk Q/I$, or even restricting to finite-dimensional ones, there can be arbitrarily many irreducible components of the representation varieties $\rep_A(\bd)$, and their singularities can be smoothly equivalent to any singularity that appears in a finite type scheme over $\ZZ$.  Thus, there is no reasonable expectation for results on the geometry of $\rep_A(\bd)$ for arbitrary $A$ and $\bd$.  One instead restricts to specific classes of algebras, and even then there are very few where singularity properties or prime ideals of irreducible components (or orbit closures) can be described.  Below we survey some literature on these problems and how our work relates.

\textit{Comparison with stratification methods:} 
Stratification and partition methods have been used quite successfully in works such as
\cite{HZ07,BHZT09,GHZ18} to study representation varieties, specifically to parametrize irreducible components and describe generic properties of representations in them.  The map we call $r_x(M)$ in Section \ref{sec:nodebundles} counts the number of simple summands of $\rad M \cap \soc M$ which are supported at $x$ when $x$ is a node, and has been used in the works cited above to stratify representation varieties.
Example \ref{ex:stratfail} illustrates a limitation of stratification methods in studying singularities of $\rep_A(\bd)$ and its closed subvarieties, motivating the additional techniques in the present work.

\smallskip

\textit{Singularities and defining equations:} The authoritative source on singularities in representation varieties
is the survey of Zwara \cite{Zwarasurvey}.
Some additional contributions to this topic can be found in more recent papers such as  \cite{Bobinski12,RZ13,LZ14,Sutar15,lor,lorwey}.  To the best of our knowledge, our results give the first classes of algebras (other than the trivial hereditary case) where every irreducible component of every representation variety is known to have rational singularities.

Particular cases of radical square zero algebras have been studied from the point of view of singularities and defining equations. In the articles \cite{Kempfcomp,DecoStrick, Strick1,Strick2, MT1,MT2} the authors study radical square zero algebras of quivers of type $\bbA, \tilde{\bbA}_1, \tilde{\bbA}_2$, where they prove that the irreducible components of the respective varieties of complexes and circular complexes (more generally, the respective varieties $C_\br$ as defined in Theorem \ref{thm:comprad}) have rational singularities, and provide generators of their defining ideals. In this paper we generalize these results to arbitrary quivers (see also Remark \ref{rem:char}).

\smallskip

\textit{Irreducible components:}
Irreducible components of representation varieties have been studied for many algebras, both in works mentioned above and also for example in \cite{BS01,GS03,Schroer04,RRS11,KW14,BS19,BS19b,Bob19}; see also \cite{GHZ12} for a survey.  
All our results on parametrization of irreducible components appeared or can be derived from earlier results in the literature, but in other language; thus we include our own treatment for completeness.  Detailed citations are given in the body of the text.

\smallskip

\textit{Moduli spaces:}
Moduli spaces were introduced into the representation theory of finite-dimensional algebras by A.D. King \cite{King94}, motivated by earlier work on moduli spaces of vector bundles such as \cite{GR87,Bondal}.  See also \cite{Hoskins} for a more recent review article.
One motivation to study normality of irreducible components of representation varieties comes from the main theorem of \cite{CKdecomp} which gives decompositions of moduli spaces of semistable representations.  
These decompositions were first proven for algebras of the form $A=\kk Q$ with $Q$ acyclic in \cite{DW11} (see also \cite{CB02}), then extended to certain classes of non-hereditary algebras in works such as \cite{Chindris10,CC15,CCKW17}.
 We discuss such applications and related semistability results in Section \ref{sec:moduli}.

\subsection*{Acknowledgements}
The authors thank Birge Huisgen-Zimmermann for discussions on radical square zero algebras and Claus Ringel for encouragement to consider the relative situation of node splitting in general.  We also thank Paul Muhly for discussion on the history of graphs, quivers, and radical square zero algebras.  Finally, we are especially grateful to the final referee for very careful readings that improved the quality of the article.
This work was supported by a grant from the Simons Foundation (636534, RK).


\section{Background}
\subsection{Quivers}
We denote a \emph{quiver} by $Q=(Q_0, Q_1, t, h)$, where $Q_0$ is the \emph{vertex} set, $Q_1$ the \emph{arrow} set, and $t, h\colon Q_1 \to Q_0$ give the \emph{tail} and \emph{head} of an arrow $t\za \xrightarrow{\za} h\za$.
A \emph{representation} $M$ of $Q$ is a collection of (finite-dimensional) $\kk$-vector spaces $(M_z)_{z \in Q_0}$ assigned to the vertices of $Q$, along with a collection of $\kk$-linear maps $(M_\alpha \colon M_{t\alpha} \to M_{h\alpha})_{\alpha \in Q_1}$ assigned to the arrows.  
We recall the some key facts here, but for a more detailed recollection we refer the interested reader to standard references such as \cite{assemetal,Schiffler:2014aa,DWbook}.

A quiver $Q$ determines a \emph{path algebra} $\kk Q$.
The category of (left) modules over the algebra $\kk Q/I$ is equivalent to the category of representations of the \emph{quiver with relations} $(Q,R)$, where $R$ is usually taken to be a minimal set of generators of $I$.
These equivalences can be used freely without significantly affecting the geometry, as made precise in \cite{Bongartz91}.

Given a nonnegative integer $n$, we write $Q_{\geq n}$ for the set of all paths of $Q$ of length greater than or equal to $n$, and $\kk Q_{\geq n} \subseteq \kk Q$ for the linear space of this set.
An ideal is \emph{admissible} if $\kk Q_{\geq N} \subset I \subset \kk Q_{\geq 2}$ for some $N \geq 2$.  
Given a finite-dimensional $\kk$-algebra $A$, it is Morita equivalent to a quotient of a path algebra $\kk Q/I$. If $I$ is taken to be admissible (which is always possible), then $Q$ is uniquely determined, 
and the Jacobson radical $\rad(\kk Q/I)$  is spanned by $Q_{\geq 1}$ modulo $I$. We always assume that $I \subset \kk Q_{\geq 2}$ throughout the paper, and that $I$ is admissible in Section \ref{sec:moduli} for our results on moduli spaces.

\subsection{Node splitting}\label{sec:nodesplit}
The operation of node splitting for Artin algebras was introduced by Mart\'inez-Villa \cite{MV80}, and further publicized in \cite[X.2]{ARS}.  Here we recall this notion in language translated to quotients of quiver path algebras.
We say that $x\in Q_0$ is a \emph{node} of an algebra $A=\kk Q/I$ if $\za \zb \in I$ for all pairs $\za, \zb \in Q_1$ such that $h\za = x$ and $t\zb = x$. 
In other words, any path having $x$ as an intermediate vertex is 0 in $A$.
Note that we allow sinks and sources to be ``nodes'' in this paper, contrary to the usual definition which omits these.
 
Given a quiver $Q$ and $x\in Q_0$, we can consider the quiver $Q^x$ with vertex set
\[
(Q^x)_0 = (Q_0 \backslash \{x\}) \cup \{x_t, x_h\}
\]
and arrow set $(Q^x)_1 = Q_1$.  Tail and head functions $t, h\colon (Q^x)_1 \to (Q^x)_0$ are the same as in $Q$ except that $x$ is replaced with $x_t$ in the codomain of $t$, and $x$ is replaced with $x_h$ in the codomain of $h$.
The operation locally around $x$ is illustrated below.
\[
Q=
\vcenter{\hbox{\begin{tikzpicture}[point/.style={shape=circle,fill=black,scale=.5pt,outer sep=3pt},>=latex]
   \node[point] (y1) at (-2,1) {};
  \node[point] (y2) at (-1,1) {};
  \node[outer sep=-2pt] at (0,1) {${\cdots}$};
  \node[point] (yr) at (1,1) {};
   \node[label={[label distance=0.5cm]180:$\cdots$}] (x) at (-1,0) {${x}$};
   \node[point] (z1) at (-2,-1) {};
  \node[point] (z2) at (-1,-1) {};
  \node[outer sep=-2pt] at (0,-1) {${\cdots}$};
  \node[point] (zs) at (1,-1) {};
\draw[->] (x) to [out=155,in=205,looseness=10] (x);
\draw[->] (x) to [out=150,in=210,looseness=20] (x);
  \path[->]
	(y1) edge  (x)
	(y2) edge  (x)
	(yr) edge  (x)
  	(x) edge  (z1)
	(x) edge  (z2)
	(x) edge  (zs);
   \end{tikzpicture}}}
\qquad\rightsquigarrow \qquad
Q^x=\quad
\vcenter{\hbox{\begin{tikzpicture}[point/.style={shape=circle,fill=black,scale=.5pt,outer sep=3pt},>=latex]
   \node[point] (y1) at (-1,1) {};
  \node[point] (y2) at (0,1) {};
  \node[outer sep=-2pt] at (1,1) {${\cdots}$};
  \node[point] (yr) at (2,1) {};
   \node[outer sep=-2pt] (xt) at (-1,0) {${x_t}$};
   \node[outer sep=-2pt] (xh) at (1,0) {${x_h}$};
   \node[point] (z1) at (-2,-1) {};
  \node[point] (z2) at (-1,-1) {};
  \node[outer sep=-2pt] at (0,-1) {${\cdots}$};
  \node[point] (zs) at (1,-1) {};
   \node at (0,.1) {$\vdots$};
  \path[->]
	(y1) edge  (xh)
	(y2) edge  (xh)
	(yr) edge  (xh)
  	(xt) edge  (z1)
	(xt) edge  (z2)
	(xt) edge  (zs);
\draw[->,bend left=20] (xt.20) to (xh.160);
\draw[->,bend right=20] (xt.-20) to (xh.200);
   \end{tikzpicture}}}
\]
Notice that $Q^x_{\geq 1}$ can be naturally identified with a subset of $Q_{\geq 1}$, inducing an inclusion of vector spaces $\kk Q^x_{\geq 1} \subset \kk Q_{\geq 1}$. 

Let $x$ be a node of an algebra $\kk Q/I$.   
The algebra $A^x=\kk Q^x/I^x$ is defined with $Q^x$ as above and $I^x = I \cap \kk Q^x_{\geq 1}$.
It is easily observed that if $x$ is a node of $A=\kk Q/I$, then $A^x$ has exactly one fewer node than $A$ (not counting sources and sinks).
Furthermore, it is not difficult to see that the category of left $A$-modules is equivalent to the full subcategory of left $A^x$-modules which have no direct summand isomorphic to the simple at $x_h$.

An algebra $A=\kk Q/I$ such that every vertex of $Q$ is a node is a \emph{radical square zero algebra}, which is equivalent to $I=\kk Q_{\geq 2}$, and equivalent to $\rad^2 A = 0$.

\subsection{Representation varieties}\label{sec:repvariety}
Given a quiver $Q$ and \emph{dimension vector} $\bd \colon Q_0 \to \mathbb{Z}_{\geq 0}$, we study the \emph{representation variety}
\[
\rep_Q(\bd) = \prod_{\za \in Q_1} \Mat(\bd(h\za), \bd(t\za)),
\]
where $\Mat(m,n)$ denotes the variety of matrices with $m$ rows, $n$ columns, and entries in $\kk$.
We consider the left action of the \emph{base change group}
\[
GL(\bd) = \prod_{z \in Q_0} GL(\bd(z))
\]
on $\rep_Q(\bd)$ given by
\[
g\cdot M = (g_{h\za}M_\alpha g_{t\za}^{-1})_{\za\in Q_1},
\]
where $g = (g_z)_{z \in Q_0} \in GL(\bd)$ and $M = (M_\za)_{\za \in Q_1} \in \rep_Q(\bd)$.

Now consider an algebra $A=\kk Q/I$ with corresponding quiver with relations $(Q,R)$. Then the representation variety $\rep_A(\bd)$ is the closed $GL(\bd)$-stable subvariety of $\rep_Q(\bd)$ defined by
\[\rep_A(\bd) = \{ M \in \rep_Q(\bd) \, \mid \, M(r)=0, \, \mbox{ for all } r\in R\}.\]
Thus, the points of $\rep_A(\bd)$ are representations of $(Q,R)$ of dimension vector $\bd$.  
Simply from the definitions, $GL(\bd)$-orbits in $\rep_A(\bd)$ are in bijection with isomorphism classes of representations of $A$ of dimension vector $\bd$. For a representation $M$ of $A$ of dimension vector $\bd$, we denote by $O_M$ the orbit in $\rep_A(\bd)$ corresponding to the isomorphism class of $M$, and by $\ol{O}_M$ the closure of this orbit.

\subsection{Rational singularities}

We say that a morphism between algebraic varieties $f:Z\to X$ is a \emph{resolution of singularities} if both $Z$ is smooth, and $f$ is proper and birational.

When $\charac\kk =0$, we say that an algebraic variety $X$ has \emph{rational singularities}, if for some (hence, any) resolution of singularities $Z\to X$, we have
\begin{itemize}
\item[(a)] $X$ is normal, that is, the natural map $\cO_X \to f_* \cO_Z$ is an isomorphism, and
\item[(b)] $\bR^i f_* \cO_Z = 0$, for $i>0$.
\end{itemize}
It is known that if $X$ has rational singularities, then $X$ is a Cohen--Macaulay variety. For more details, we refer the reader to \cite[Section 1.2]{weymanbook}.

\subsection{Homogeneous fiber bundles}\label{sec:bundles}
Let $G$ be an algebraic group and $H\leq G$ a closed algebraic subgroup, and suppose we have an action of $H$ on a quasi-projective algebraic variety $S$.
We write $G \times_H S$ for the quotient of $G \times S$ by the free left action of $H$ given by $h\cdot (g,s) = (gh^{-1}, h\cdot s)$, called an \emph{induced space} or \emph{homogeneous fiber bundle}.
We consider this quotient as a $G$-variety by the action $g\cdot (g',s) = (gg',s)$. 
Furthermore, we embed $S \into G \times_H S$ via the map $s \mapsto (1,s)$.
The following lemma can be proven directly from definitions; see for example \cite[\S2.1]{Timashev} for further discussion.

\begin{lemma}\label{lem:fiberbundle}
The maps below are mutually inverse, inclusion preserving bijections.
\[
\begin{split}
\left\{\begin{tabular}{c} $G$-stable subvarieties\\ of $G \times_H S$\end{tabular} \right\}
& \leftrightarrow
\left\{\begin{tabular}{c}  $H$-stable subvarieties\\ of $S$ \end{tabular} \right\}
\\
Y \qquad \qquad &\mapsto \qquad \qquad Y \cap S\\
G\times_H Z \qquad \qquad &\mapsfrom \qquad \qquad Z
\end{split}
\]
In particular, they give a bijection on orbits, and an isomorphism of orbit closure posets.
\end{lemma}


\section{Node splitting and bundles}\label{sec:nodebundles}
Consider an algebra $A=\kk Q/I$, and $\bd$ a dimension vector of $Q$. 

\subsection{Node splitting on strata}
Throughout we use the following notation.

\begin{definition}\label{def:xrank}
For $x\in Q_0$ and $M\in\rep_A(\bd)$, we define linear maps $h_x(M)$ and $t_x(M)$ by
\begin{equation}\label{eq:hxtx}
h_x(M)=\bigoplus_{h\za =x} M_{\za} \colon  \, \bigoplus_{h\za =x} M_{t\za} \to M_x, \, \mbox{ and } \, t_x(M)=\bigoplus_{t\za =x} M_{\za}  \, \colon M_{x} \to \bigoplus_{t\za =x} M_{h\za}.
\end{equation}
 Given a non-empty subset $S\subset \rep_A(\bd)$, we define the $x$-\textit{rank} of $S$ to be the number
\[r_x(S) := \max_{M \in S} \left\{\rank{h_x(M)} \right\}.\]
Moreover, we denote by $S^\circ_x = \{M \in S \,\mid\, r_x(M)=r_x(S)\}$, where $r_x(M):=r_x(\{M\})$.
\end{definition}


Now assume that $x\in Q_0$ is a node of $A$. Let $A^x= \kk Q^x/I^x$ be the algebra obtained by splitting the node $x$, as explained in Section \ref{sec:nodesplit}. 
Fix an integer $r$ with $0\leq r \leq \bd(x)$. We denote by $\bd^x_r$ the dimension vector of $Q^x$ obtained by putting $\bd^x(x_h) = r$, $\bd^x(x_t) = \bd(x) - r$, and at the rest of the vertices $\bd^x$ coincides with $\bd$. We can view $GL(\bd^x_r)$ naturally as a subgroup of $GL(\bd)$ via diagonal embedding at $x$.
We realize the variety $\rep_{A^x}(\bd^x_r)$ as a $GL(\bd^x_r)$-stable closed subvariety of $\rep_A(\bd)$ by an embedding $i\colon \rep_{A^x}(\bd^x_r) \into \rep_A(\bd)$.  Namely, given $M=(M_\za)_{\za \in Q_1}$ in the domain, we take $i(M) = (N_\za)_{\za \in Q_1}$ where
\begin{equation}\label{eq:repAxembed}
N_\za = 
\begin{cases}
M_\za & t\za \neq x \neq h\za\\
\left[\begin{smallmatrix}M_\za \\ 0\end{smallmatrix}\right] & h\za =x \mbox{ and } t\za \neq x,\\
\left[\begin{smallmatrix}0 & M_\za \end{smallmatrix}\right] & t\za =x \mbox{ and } h\za \neq x,\\
\left[\begin{smallmatrix}0 & M_\za \\ 0 & 0 \end{smallmatrix}\right] & t\za =x \mbox{ and } h\za = x.\\
\end{cases}
\end{equation}
In the remainder of the paper we implicitly use this specific embedding $\rep_{A^x}(\bd^x_r) \subset \rep_A(\bd)$ without mentioning the map $i$.

If $C$ is a $GL(\bd^x_r)$-stable irreducible closed subvariety of $\rep_{A^x}(\bd^x_r)$ (hence of $\rep_A(\bd)$), we take its $GL(\bd(x))$-saturation to obtain the $GL(\bd)$-stable subset $GL(\bd(x)) \cdot C$ of $\rep_A(\bd)$. Note that we have $r_x(GL(\bd(x)) \cdot C) = r_{x_h}(C) \leq r$.

Retaining the notation above, furthermore let $P_r\leq GL(\bd(x))$ be the parabolic subgroup of block upper triangular matrices with two blocks along the diagonal, of size $r$ in the upper left and $\bd(x)-r$ in the lower right.
From \eqref{eq:repAxembed} we see that $\rep_{A^x}(\bd^x_r)$ is not only $GL(\bd^x_r)$-stable, but also a $P_r$-stable subvariety of $\rep_A(\bd)$, since the unipotent radical of $P_r$ acts trivially on $\rep_{A^x}(\bd^x_r)$.

\begin{remark}
In our setup above, let $P^x_r(\bd) \leq GL(\bd)$ be the subgroup where the factor $GL(\bd(x))$ is replaced by $P_r$, so we have also that $P^x_r(\bd) \geq GL(\bd_r^x)$.
We observe here that 
$GL(\bd(x))\times_{P_r} C= GL(\bd)\times_{P^x_r(\bd)} C$ 
for any $GL(\bd^x_r)$-stable irreducible closed subvariety $C\subset \rep_{A^x}(\bd^x_r)$, and the saturations $GL(\bd(x)) \cdot C$ and $GL(\bd) \cdot C$ are the same.
We use this observation (particularly when applying Lemma \ref{lem:fiberbundle}) throughout the article without always explicitly mentioning it.
\end{remark}

\begin{proposition}\label{prop:xrbundle}
Let $0 \leq r \leq \bd(x)$ and consider the locally closed subvariety of $\rep_A(\bd)$ consisting of all points of $x$-rank exactly $r$:
\[\rep^r_A(\bd) := \setst{N \in \rep_A(\bd)}{r_x(N)=r}.\]
Then $\rep^r_A(\bd)$ is non-empty if and only if $r_x( \rep_{A^x}(\bd^x_r)) =r$, in which case the following map is an isomorphism of varieties:
\[\Psi\colon \, GL(\bd(x))\times_{P_r} \rep_{A^x}(\bd^x_r)^\circ_x \, \longrightarrow \, \rep_A^r(\bd) \, , \quad (g,M) \mapsto g\cdot M.\]
\end{proposition}
\begin{proof}

The map is well-defined since $\Psi(gp^{-1}, pM) = \Psi(g, M)$. To construct the inverse morphism, take any $N \in \rep_A^r(\bd)$, so we know $ \im h_x(N)$ is an $r$-dimensional subspace of $\kk^{\bd(x)}$. Then we can find $g \in GL(\bd(x))$ such that 
\[
\im h_x(g^{-1}\cdot N) =g^{-1}_x \left(\im h_x(N)\right)=  \kk^r \subset \kk^{\bd(x)},
\]
the subspace spanned by the first $r$ standard basis vectors.  Since $x$ is a node in $A$, this means $g^{-1}\cdot N \in \rep_{A^x}(\bd^x_r)$ via the identification of \eqref{eq:repAxembed}.  This $g$ is not unique, but any $g_0 \in GL(\bd(x))$ with the same property satisfies that $g_0^{-1}g$ stabilizes $\kk^r$, 
which is to say $g_0^{-1}g \in P_r$. Thus $(g, g^{-1}\cdot N)$ and $(g(g_0^{-1}g)^{-1}, (g_0^{-1}g)g^{-1}\cdot N) = (g_0, g_0^{-1}\cdot N)$ represent the same point in the quotient variety 
$GL(\bd(x))\times_{P_r} \rep_{A^x}(\bd^x_r)^\circ_x$, so $N \mapsto (g, g^{-1} \cdot N)$ as above gives a well-defined inverse morphism to $\Psi$.
\end{proof}

The following well-known result of Mart\'inez-Villa \cite{MV80} follows quickly from Proposition \ref{prop:xrbundle}.
We include a short proof to illustrate the use of geometric methods in representation theory of algebras.

\begin{corollary}\label{cor:indec}
There is a bijection between the set of isomorphism classes of indecomposable representations of $A$, and the set of isomorphism classes of indecomposable representations of $A^x$ which contain no simple direct summand supported at $x_h$.
\end{corollary}

\begin{proof}
It is immediate from Proposition \ref{prop:xrbundle} that for any $\bd$ and $0 \leq r \leq \bd(x)$ we have a bijection
\begin{equation}\label{eq:orbits}
\left\{\begin{tabular}{c} isomorphism classes of \\ representations $M$ of $\rep_{A^x}(\bd^x_r)$ \\ with $r_{x_h}(M)=r$ \end{tabular} \right\}
\leftrightarrow
\left\{\begin{tabular}{c} isomorphism classes of \\ representations $N$ of $\rep_{A}(\bd)$\\ with $r_x(N)=r$ \end{tabular} \right\}.
\end{equation}
Let $S_{x_h}$ denote the simple supported at $x_h$. Clearly, if $M \in \rep_{A^x}(\bd^x_r)$ with $r_{x_h}(M)<r$, then $S_{x_h}$ is a summand of $M$. This shows that the only indecomposable representation of $A^x$ that does not appear in the sets on the left hand side of \eqref{eq:orbits} is $S_{x_h}$.

We are left to show that under the correspondence in \eqref{eq:orbits}, indecomposable representations are mapped to indecomposable representations. We use the well-known fact that a representation is indecomposable if and only if its stabilizer in the projective linear group is unipotent \cite[Cor. 2.10]{Brion-notes}. 
Let $H$ be the $PGL(\bd_x^r)$-stabilizer of $M \in \rep_{A^x}(\bd_r^x)$. Since  $GL(\bd) \cdot O_M\cong GL(\bd(x))\times_{P_r} O_M$, the $PGL(\bd)$-stabilizer of $M$ in $\rep_A(\bd)$ is $U \rtimes H$, where $U$ is the unipotent radical of $P_r\cong U \rtimes (GL(r) \times GL(\bd(x)-r))$. Clearly, $H$ is unipotent if and only if $U\rtimes H$ is.
\end{proof}

A representation $M$ of $A$ is called \emph{Schur} if $\End_A(M) = \kk$. A Schur representation is indecomposable. The following corollary (also well-known in representation theory) shows, in particular, that a Schur representation of $A^x$ does not necessarily correspond to a Schur representation of $A$ via the bijection in Corollary \ref{cor:indec}. 
\begin{corollary}\label{cor:schur}
If $M$ is a Schur representation of $A$, then either $h_x(M)=0$ or $t_x(M)=0$. Moreover, the bijection in \ref{cor:indec} descends to a bijection
\[
\left\{\begin{tabular}{c} isomorphism classes of \\ Schur representations $M\neq S_{x_h}$ of $A^x$ \\with either $M_{x_h}=0$ or $M_{x_t}=0$ \end{tabular} \right\}
\leftrightarrow
\left\{\begin{tabular}{c} isomorphism classes of \\ Schur representations of $A$ \end{tabular} \right\}.
\]
\end{corollary}

\begin{proof}
We use the notation as in the proof of Corollary \ref{cor:indec}. A representation is Schur if and only its stabilizer in the projective linear group is trivial. If $H$ is the $PGL(\bd_x^r)$-stabilizer of $M\neq S_{x_h}$ in $\rep_{A^x}(\bd_r^x)$, then $U \rtimes H$ is the $PGL(\bd)$-stabilizer of $M\in\rep_A(\bd)$. Hence $M\in \rep_A(\bd)$ is Schur if and only if $M\in \rep_{A^x}(\bd_r^x)$ is Schur and $U$ is trivial. But $U$ is trivial if and only if $r=r_x(M)$ is equal to either $0$ or $\bd(x)$, hence the claim.
\end{proof}

The following example illustrates a limitation of recovering geometric information about the whole representation variety by studying strata, motivating the use of collapsing maps in the next section.

\begin{example}\label{ex:stratfail}
Let $A=\kk[t]/(t^2)$ (i.e. the path algebra of the 1-loop quiver modulo the radical square).  Then $\rep_A(2)$ is a singular variety of dimension 2 defined by 2 equations:
\begin{equation}
\rep_A(2) = \setst{\left[\begin{smallmatrix}a & b \\ c & d\end{smallmatrix}\right]}{a+d=0,\ ad-bc=0}.
\end{equation}
Splitting the node yields $A^x \simeq \kk(\bullet \to \bullet)$, and an equivalence of categories $A$-$\module \simeq \cT \subset A^x$-$\module$, where $\cT$ is the full subcategory whose objects have no direct summand isomorphic to the simple at the sink.
Partitioning the orbits by dimension of the top of the corresponding module (equivalently, dimension of the kernel of the action of $t$), this induces a bijection:
\begin{equation}
\left\{\text{orbits in }\rep_A(2)\right\} \longleftrightarrow \left\{\text{non-zero orbit in }\rep_{A^x}((1,1))\right\} \coprod \rep_{A^x}((2,0)).
\end{equation}
Notice that the variety $\rep_A(2)$ containing all orbits appearing on the left hand side is singular of dimension 2, while the varieties for $A^x$ containing the orbits in each piece of the right hand side are smooth of dimensions 1 and 0.
\end{example}

\subsection{Passage to closed subvarieties}\label{sec:closed}
The stratifications of the previous section generally lose information about singularities of closed subvarieties, as seen in Example \ref{ex:stratfail}.  However, ``collapsing maps'' as in the work of Kempf \cite{Kempf76} can be used to partially rectify this.

\begin{proposition}\label{prop:birational}
Let $0 \leq r \leq \bd(x)$ and $C$ a $GL(\bd^x_r)$-stable irreducible closed subvariety of $\rep_{A^x}(\bd^x_r)$ with $r_{x_t}(C)=r$. Then $GL(\bd(x)) \cdot C$ is an irreducible closed subvariety of $\rep_A(\bd)$, and the following map is a proper birational morphism of $GL(\bd)$-varieties:
\[\Psi_C\colon \, GL(\bd(x))\times_{P_r} C \, \longrightarrow \, GL(\bd(x)) \cdot C\, , \quad (g,M) \mapsto g\cdot M.\]
\end{proposition}
\begin{proof}
The Grassmannian $GL(\bd(x))/P_r \cong \Gr(r, \kk^{\bd(x)})$ is a projective variety. We have an isomorphism of varieties $GL(\bd(x))\times_{P_r} \rep_A(\bd) \cong GL(\bd(x))/P_r \times \rep_A(\bd)$ given by the map $(g,x)\mapsto (g,gx)$. Hence, the multiplication map 
\[GL(\bd(x))\times_{P_r} \rep_A(\bd) \to \rep_A(\bd) \, , \,\, (g,M) \mapsto g\cdot M \] 
is proper.  Since $C$ is closed in $\rep_A(\bd)$, it then follows by Lemma \ref{lem:fiberbundle} that $GL(\bd(x)) \cdot C$ is closed in $\rep_A(\bd)$ as well. By Proposition \ref{prop:xrbundle}, the map $\Psi_C$ induces an isomorphism on the open subsets
\[GL(\bd(x))\times_{P_r} C^\circ_x \xrightarrow{\cong} (GL(\bd(x)) \cdot C)^\circ_x.\]
Hence, $\Psi_C$ is birational.
\end{proof}

\begin{proposition}\label{prop:bijection}
For each $0 \leq r \leq \bd(x)$, the maps below are mutually inverse, inclusion preserving bijections.
\[
\begin{split}
\left\{\begin{tabular}{c} irreducible closed\\ $GL(\bd^x_r)$-stable subvarieties\\ of $\rep_{A^x}(\bd^x_r)$ of $x_h$-rank $r$ \end{tabular} \right\}
&\leftrightarrow
\left\{\begin{tabular}{c} irreducible closed\\ $GL(\bd)$-stable subvarieties\\ of $\rep_{A}(\bd)$ of $x$-rank $r$ \end{tabular} \right\}
\\
C \qquad \qquad &\mapsto \qquad \qquad GL(\bd(x)) \cdot C\\
D \, \cap \, \rep_{A^x}(\bd_r^x) \qquad \qquad &\mapsfrom \qquad \qquad D
\end{split}
\]
\end{proposition}

\begin{proof}

It follows by Proposition \ref{prop:birational} that $C\mapsto GL(\bd(x)) \cdot C$ is a well-defined function between the sets above. 
Each subvariety $C$ (resp $D$) in the set on the left (resp. right) hand side above is uniquely determined by $C^\circ_x$ (resp $D^\circ_x$) via $C=\ol{C^\circ_x}$ (resp. $D=\ol{D^\circ_x}$), hence the map $C\mapsto GL(\bd(x)) \cdot C$ is bijective by Proposition \ref{prop:xrbundle} and Lemma \ref{lem:fiberbundle}.

To show that the inverse map is the one claimed, we are left to show that 
$$(GL(\bd(x)) \cdot C) \, \bigcap \, \rep_{A^x}(\bd_r^x) = C$$
 (in fact, this holds for any $GL(\bd^x_r)$-stable subset $C$ of $\rep_{A^x}(\bd^x_r)$).  
The containment $\supseteq$ is immediate, so we must show the other direction. 
 Take $g\in GL(\bd(x))$ and $M\in C$ such that $g \cdot M \in \rep_{A^x}(\bd^x_r)$. We want to show that $g\cdot M \in C$. To do so, it is enough to find $g' \in GL(r)\times GL(\bd(x)-r)$ such that $g'\cdot M = g \cdot M$ since $C$ is $GL(\bd^x_r)$-stable.  Such a $g'$ exists if and only if $M$ is isomorphic to $g \cdot M$ when considered as a representation of $A^x$.
So this containment is essentially just saying that if two representations of $A$ are isomorphic (by $g$), then they are isomorphic when considered as representations of $A^x$ (by $g'$).
Let $B_1$ (resp. $B_2$)  be the matrix of the map $h_x(M)$ (resp. $h_x(g\cdot M)$). Since $M,\, g \cdot M \in \rep_{A^x}(\bd^x_r)$, the images of both $\alpha_1, \alpha_2$ are contained in $\kk^r$. Hence, only the first $r$ rows of $B_1$ (resp $B_2$) are non-zero. We denote the matrix formed by the first $r$ rows of $B_1$ (resp. $B_2$) by $B_1'$ (resp. $B_2'$). Since $B_1$ and $B_2$ are row-equivalent (i.e. have the same reduced row echelon form), the matrices $B_1'$ and $B_2'$ are also row-equivalent. Using the same argument with the maps with source $x$, we see that there is a matrix $g' \in GL(r)\times GL(\bd(x)-r)$ such that $g' \cdot M = g \cdot M$. 
\end{proof}

The following two corollaries are immediate from Proposition \ref{prop:bijection}. 
The first can also be easily obtained from stratification methods, but we record it here for completeness of our exposition.

\begin{corollary}\label{cor:irredcomp}
There is an injective map of sets
\begin{equation}\label{eq:irredcomp}
\coprod_{\bd} \left\{\begin{tabular}{c} irreducible components\\ of $\rep_{A}(\bd)$ \end{tabular} \right\}
\into
\coprod_{\be} \left\{\begin{tabular}{c} irreducible components\\ of $\rep_{A^x}(\be)$ \end{tabular} \right\}
\end{equation}
from the set of all irreducible components of all $\rep_A(\bd)$ to the set of all irreducible components of all $\rep_{A^x}(\be)$.
\end{corollary}

\begin{corollary}\label{cor:orbclos}
Let $N\in \rep_A^r(\bd)$, and consider its $GL(\bd)$-orbit $O_N$ in $\rep_A(\bd)$. Take any $M\in O_N \, \bigcap \, \rep_{A^x}(\bd^x_r)$ (such element exists by  \eqref{eq:orbits}). Then  $\ol{O}_N=GL(\bd(x)) \cdot \ol{O}_M$, where $\ol{O}_M$ denotes the closure of the $GL(\bd_r^x)$-orbit $O_M$ in $\rep_{A^x}(\bd^x_r)$.
\end{corollary}


\subsection{Irreducible components}
The results of this section can be derived from the methods used in \cite{BCHZ15} and \cite{HZS17} to study radical square zero algebras and truncated path algebras, respectively; see below for more specific references.  
We include our own proofs and examples for completeness, and to emphasize the ``relative setting'' we work in:
namely, repeated application of Corollary \ref{cor:irredcomp} can be used to classify irreducible components of $\rep_A(\bd)$ if splitting the nodes eventually results in \emph{any} representation variety whose irreducible components are known.

\begin{example}\label{ex:irrcomp}
Consider the algebra $A=\kk Q/I$ where $Q$ is given below and $I$ is generated by the 12 relations necessary to make $x$ a node, along with the relation $cba=0$.
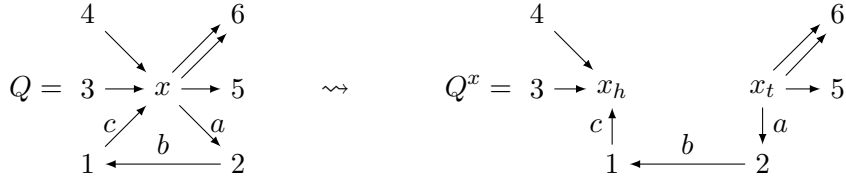
\[
Q=\vcenter{\hbox{\begin{tikzpicture}[point/.style={shape=circle,fill=black,scale=.5pt,outer sep=3pt},>=latex]
  \node (x) at (0,1) {$x$};
  \node (1) at (-1,0) {1};
  \node (2) at (1,0) {2};
  \node (3) at (-1,1) {3};
  \node (4) at (-1,2) {4};
  \node (5) at (1,1) {5};
  \node (6) at (1,2) {6};
  \path[->]
	(4) edge (x)
	(3) edge (x) 
	(1) edge node[left] {$c$} (x)
	(2) edge node[above] {$b$} (1)
	(x.60) edge (6.210)
          (x.30) edge (6.240)
	(x) edge (5)
	(x) edge node[right] {$a$} (2);
   \end{tikzpicture}}}
\qquad\rightsquigarrow \qquad
\quad
Q^x=\vcenter{\hbox{\begin{tikzpicture}[point/.style={shape=circle,fill=black,scale=.5pt,outer sep=3pt},>=latex]
  \node (xh) at (-1,1) {$x_h$};
  \node (xt) at (1,1) {$x_t$};
  \node (1) at (-1,0) {1};
  \node (2) at (1,0) {2};
  \node (3) at (-2,1) {3};
  \node (4) at (-2,2) {4};
  \node (5) at (2,1) {5};
  \node (6) at (2,2) {6};
  \path[->]
	(4) edge (xh)
	(3) edge (xh) 
	(1) edge node[left] {$c$} (xh)
	(2) edge node[above] {$b$} (1)
	(xt.60) edge (6.210)
          (xt.30) edge (6.240)
	(xt) edge (5)
	(xt) edge node[right] {$a$} (2);
   \end{tikzpicture}}}
\]
Notice that $A$ does not fall within a well-studied class such as special biserial, radical square zero, etc.  
The overlapping relations make direct analysis of irreducible components of $\rep_A(\bd)$ challenging. Also, $A$ is not representation finite, so irreducible components cannot be determined by computing dimensions of Hom spaces as in \cite{Zwara99}. But each $\rep_{A^x}(\bd^x_r)$ decomposes as the product of an affine space with a union of orbit closures (determined by the relation $cba$) in a representation variety for the subquiver of Dynkin type $\bbA_4$ with arrows $a, b,c$.  So these can be explicitly determined for any given $\bd^x_r$.

For example, take $\bd=(3,2,2,1,3,3,3)$ (with the convention that $\bd(x)$ is the last entry). By Proposition \ref{prop:bijection}, the irreducible components of $\rep_A(\bd)$ are among the $GL(\bd)$-saturations of the irreducible components of $\rep_{A^x}(\bd^x_r)$ for $r=0,1,2,3$, which reduce to the quiver of type $\bbA_4$ with the following dimension vector
\[\begin{tikzpicture}[point/.style={shape=circle,fill=black,scale=.5pt,outer sep=3pt},>=latex]
  \node (1) at (0,0){$(3-r)$};
  \node (2) at (2,0) {2};
  \node (3) at (4,0) {3};
  \node (4) at (6,0) {$r$};
  \path[->]
	(1) edge node[above] {$a$} (2)
	(2) edge node[above] {$b$} (3)
	(3) edge node[above] {$c$} (4);
   \end{tikzpicture}
\]
We first describe the irreducible components of the $\bbA_4$ quiver above, with the convention that indecomposables correspond to roots (their dimension vectors). In the cases $r=0$ and $r=3$, the representation varieties are irreducible affine spaces $C_0, C_3$. When $r=1$, the representation variety has two components $C_1$ and $C_1'$, which are the orbit closures of the representations $(1,1,1,0)^{\oplus 2} \oplus (0,0,1,1)$ and $(1,0,0,0)\oplus(1,1,1,0) \oplus (0,1,1,1) \oplus (0,0,1,0)$, respectively. For $r=2$, there are again two components $C_2$ and $C_2'$, which are the closures of  $(1,1,1,0)\oplus(0,1,1,1)\oplus(0,0,1,1)$ and $(1,0,0,0)\oplus (0,1,1,1)^{\oplus 2} \oplus (0,0,1,0)$, respectively.

By abuse of notation, we use the same symbols for the components for $A^x$ obtained from the components above. 
Since all the components have maximal $x_h$-rank, their $GL(\bd)$-saturations yield irreducible closed subsets in $\rep_A(\bd)$ according to Proposition \ref{prop:bijection}. By looking at generic ranks of matrices along each non-trivial path in $Q^x$, we see that there are no inclusions among the saturations of irreducible components. Here we must use all paths in $Q^x$ and not just those in the $\mathbb{A}_4$ subquiver: for example, the ranks of all paths in the $\mathbb{A}_4$ subquiver for points in $C_0$ are less than those ranks for points in $C_1$, but on the other hand, the rank of the map over the arrow $x\to 5$ is generically 3 on $C_0$ but generically 2 on $C_1$, so we cannot have $C_0$ contained in $C_1$.
Thus $\rep_A(\bd)$ has $6$ irreducible components, which are given by the $GL(\bd)$-saturations of $C_0,C_1,C_1',C_2,C_2',C_3$. 
\end{example}

\medskip

In Proposition \ref{prop:bijection} the irreducible components of $\rep_{A^x}(\bd^x_r)$ do not necessarily yield irreducible components in $\rep_A(\bd)$ under the map $C\to GL(\bd(x)) \cdot C$. Nevertheless, we give a condition when this indeed happens.

\begin{lemma}\label{lem:comp}
Let $C$ be an irreducible component of $\rep_{A^x}(\bd^x_r)$ with $r_{x_h}(C)=r$, and assume that there is a representation $M\in C$ such that the map
$t_{x_t}(M)$ is injective. Then $GL(\bd(x)) \cdot C$ is an irreducible component of $\rep_A(\bd)$.
\end{lemma}

\begin{proof}
Assume by contradiction that $GL(\bd(x)) \cdot C$ is strictly contained in an irreducible component $C'$ of $\rep_A(\bd)$. By Proposition \ref{prop:bijection}, we must have $r_x(C')>r$. On the open subset of $C'$ of representations $N$ with $r_x(N)=r_x(C')$, we must have $\dim \ker t_x(N)>r$. Since $C'$ is irreducible, this shows that for all representations $N\in C'$, we have $\dim \ker t_x(N)>r$. But the assumptions imply that for $M$ (viewed as a representation in $\rep_A(\bd)$), we have $\dim \ker t_x(M)=r$. Hence $M\notin C'$, a contradiction.
\end{proof}

Applying recursively Proposition \ref{prop:bijection}, we give an explicit description of the irreducible components of representation varieties for radical square zero algebras.
All parts of the theorem below appeared elsewhere in various forms, where our $\br$ corresponds to the dimension vector of the semisimple module $T$:
\begin{itemize}[leftmargin=1em]
\item Part (1) is a consequence of the more general results on varieties of representations with a fixed radical layering for truncated path algebras in \cite{BHZT09} (Theorem 5.3 by way of Proposition 2.2);
\item Part (2) is found in \cite[Definition and Comments 3.3]{BCHZ15};
\item Part (3) is equivalent to the representation-theoretic version \cite[Proposition 3.9]{BCHZ15}.
\end{itemize}
We include a more geometric proof as it implicitly also constructs a resolution of singularities of each $C_\br$, which is used for example in \cite[Corollary 4.2]{collapsechar}.
\begin{theorem}\label{thm:comprad}
Consider a radical square zero algebra $A=\kk Q/\kk Q_{\geq 2}$ and a dimension vector $\bd$. For a dimension vector $\br \leq \bd$, we denote by $C_\br$ the closure of the set of representations $M \in \rep_A(\bd)$ such that $r_x(M)=\br(x)$, for all $x\in Q_0$. Furthermore, set $\bs=\bd-\br$, and for $x\in Q_0$  let $l_x$ be the number of loops at $x$ and put 
\[u_x(\br)= \sum_{h\za = x} \bs(t\za) \, -\br(x) , \quad \mbox{ and } \quad  v_x(\br)=\sum_{t\za =x} \br(h\za) \, -\bs(x).\]
Then we have the following:
\begin{enumerate}[leftmargin=2em]
\item $C_\br$ is irreducible;
\item $C_\br$ is non-empty if and only if $u_x(\br)\geq 0$ for all $x\in Q_0$;
\item $C_\br$ is an irreducible component of $\rep_A(\bd)$ if and only if it is non-empty and $v_x(\br)\geq 0$ for all $x \in Q_0$ with $u_x(\br)>l_x$. Moreover, all irreducible components of $\rep_A(\bd)$ are of such form.
\end{enumerate}
\end{theorem}

\begin{proof}
Let $Q^{\op{sp}}$ denote the quiver obtained by splitting all the vertices of $A$. Clearly, $|Q^{\op{sp}}_0| = 2 | Q_0|$, and the vertices of $Q^{\op{sp}}$ are sinks $x_h$ and sources $x_t$ corresponding to the vertices $x\in Q_0$. Since the quiver $Q^{\op{sp}}$ has no relations, all of its representation varieties are irreducible affine spaces.

Fix $\br \leq \bd$ and let $\bs=\bd-\br$. For (1), we can assume that $C_\br$ is non-empty. To show that $C_\br$ is irreducible, it is enough to show that $C_\br^\circ:= \displaystyle\bigcap_{x \in Q_0} (C_\br)^\circ_x$ is so. Starting with the representation variety $\rep_A(\bd)$ and splitting the nodes of $A$ repeatedly w.r.t. the ranks given by the dimension vector $\br$, we arrive at the representation variety $\rep_{Q^{\op{sp}}}(\be)$, where $\be$ is the dimension vector given by $\be(x_h) = \br(x)$ and $\be(x_t)=\bs(x)$ for $x\in Q_0$. Via the isomorphisms in Proposition \ref{prop:xrbundle} (applied recursively), $C_\br^\circ$ corresponds to the open subset of representations $N \in \rep_{Q^{\op{sp}}}(\be)$ such that $r_{x_h}(N) = \be(x_h) = \br(x)$ for all $x\in Q_0$. Since the latter is irreducible, this shows that $C_\br^\circ$ is irreducible as well. Moreover, under the bijections in Proposition \ref{prop:bijection} (applied recursively), $C_\br$ corresponds to $\rep_{Q^{\op{sp}}}(\be)$.

Now consider part (2). Given $x\in Q_0$, it is easy to see that there is a representation $N \in \rep_{Q^{\op{sp}}}(\be)$ such that $r_{x_h}(N) = \br(x)$ if and only if $u_x(\br) \geq 0$. Since $\rep_{Q^{\op{sp}}}(\be)$ is irreducible, we obtain by Proposition \ref{prop:xrbundle} applied as above that $C_\br$ is non-empty if and only if $u_x(\br) \geq 0$ for all $x\in Q_0$.

Now take $C_\br$ non-empty in part (3). We show that if $\br$ satisfies $u_x(\br) > l_x$ and $v_x(\br)<0$ for some $x\in Q_0$, then $C_\br \subset C_{\br'}$, where $\br'(x)=\br(x)+1$, and $\br'(y)=\br(y)$ for $y\in Q_0 \setminus\{x\}$. First, we show that $C_{\br'}$ is non-empty. We have $u_x(\br')= u_x(\br)-l_x-1 \geq 0$. Since $v_x(\br)<0$, we have in particular that $\bs(x)>\br(y)$ for any vertex $y\neq x$ such that there is an arrow from $x$ to $y$, and so $u_y(\br')\geq \bs'(x)-\br'(y) = \bs(x)-1-\br(y) \geq 0$. We obtain that $u_z(\br')\geq 0$ for all $z\in Q_0$, hence $C_{\br'}$ is non-empty and it corresponds to $\rep_{Q^{\op{sp}}}(\be')$ via the bijections in Proposition \ref{prop:bijection}. Now take any $M\in C_\br$. By abuse of notation, we can view $M$ as a representation in $\rep_{Q^{\op{sp}}}(\be)$. We see that since $v_x(\br)<0$, we have a decomposition $M\cong N \oplus S_{x_t}$, where $S_{x_t}$ denotes the simple at $x_t$. Hence, as representations of $A$, we have $M\cong N \oplus S_x$. But then $N\oplus S_{x_h} \in \rep_{Q^{\op{sp}}}(\be')$ which shows that $M \cong N\oplus S_x \in C_{\br'}$ as well. Hence, $C_\br \subset C_{\br'}$. We have showed that 
\[\rep_A(\bd) = \bigcup_{\substack{\br \leq \bd \\ \br \mbox{ \footnotesize as in part \normalsize } (3)}} C_{\br}.\]
We are left to show that there are no containments between the irreducibles above. Take $C_\br,C_{\br'}$ two irreducibles from the union above, and assume that $C_\br \subset C_{\br'}$. Clearly, we must have $\br \leq \br'$, and so $\bs'\leq \bs$. Assume that there is a vertex $x\in Q_0$ such that $\br(x)<\br'(x)$. Since $u_x(\br)> u_x(\br')+l_x \geq l_x$, we must have $v_x(\br)\geq 0$. Then there is a representation $M\in\rep_{Q^{\op{sp}}}(\be)$ such that map $t_{x_t}(M)$ is injective. We conclude as in Lemma \ref{lem:comp} that $M\notin C_{\br'}$, a contradiction.
\end{proof}

\begin{remark}
The generic decomposition is a generalization to quivers with relations of Kac's canonical decomposition \cite{Kac80,Kac82} for quiver representations. It was further studied in \cite{Schofield92,C-BS,DW02,DW11}.  
It is based on the geometric Krull-Schmidt theorem of Crawley-Boevey and Schr{\"o}er \cite[Theorem 1.1]{C-BS} (see also de la Pe\~na's \cite[Lemma~1.3]{delaP}).

For quivers with relations, Babson, Huisgen-Zimmerman, and Thomas have studied generic behavior of representations in irreducible components in \cite{BHZT09}, obtaining the sharpest results for truncated path algebras.
Carroll has given a combinatorial method of producing the generic decomposition for acyclic gentle algebras in \cite{Carroll15}.

In \cite[Theorem~5.6]{BCHZ15}, Bleher, Chinburg, and Huisgen-Zimmermann describe the generic decomposition of each irreducible component $C$ for a radical square zero algebra in terms of the Kac canonical decomposition of the representation variety for the quiver associated to $C$ by splitting all nodes.  We remark here that our methods yield a relative version of this for splitting one node of any algebra with a node.

We recall that an algebra has the dense orbit property in the sense of \cite{CKW}, if each irreducible component of each of its representation varieties has a dense orbit. Since it is enough to check this property on indecomposable irreducible components, an easy consequence of the considerations above is that an algebra $A^x$ has the dense orbit property if and only if $A$ has the dense orbit property.  This implies, for example, that a radical square zero algebra has the dense orbit property if and only if it is already representation finite \cite[Theorem~7.2]{BCHZ15}. 
\end{remark}

We illustrate this in the following example where splitting nodes yields a gentle algebra. 

\begin{example}\label{ex:gendecomp}
Consider the algebra given by quiver
\[
Q=\quad\vcenter{\hbox{\begin{tikzpicture}[point/.style={shape=circle,fill=black,scale=.5pt,outer sep=3pt},>=latex]
  \node (1) at (0,2) {1};
  \node (2) at (2,2) {2};
  \node (3) at (2,1) {3};
  \node (4) at (0,1) {4};
  \path[->]
	(1.20) edge (2.160)
	(2.200) edge (1.-20);
  \path[->]
	(4.20) edge (3.160)
	(3.200) edge (4.-20);
  \path[->]
	(1.-70) edge (4.70)
	(4.110) edge (1.250);
  \path[->]
	(2.-70) edge (3.70)
	(3.110) edge (2.250);
   \end{tikzpicture}}}
\]
with $I \subset \kk Q$ generated by relations such that vertices 1 and 3 are nodes, and additionally all 2-cycles are zero.
Splitting both of the nodes 1 and 3 yields the gentle algebra of \cite[Example~1]{Carroll15}.  Irreducible components for representation varieties of this gentle algebra can be parametrized by maximal rank sequences. Carroll gives a combinatorial method for determining the generic decomposition into string and band representations for each such irreducible component, and thus the remark above gives the generic decomposition for the corresponding irreducible component of the non-gentle algebra in this example.
\end{example}


\section{Main results}\label{sec:mainresults}
Assume $\charac \kk = 0$ throughout this entire section (but see Remark \ref{rem:char}).  Here we apply the initial work above to obtain results about singularities of equivariant closed subvarieties, as well as generators of their defining ideals.

\subsection{Singularities}
We continue with a fixed $x\in Q_0$ which is a node of $A$, and $A^x$ is the algebra obtained by splitting the node $x$ as in Section \ref{sec:nodesplit}. We now prove Theorem \ref{thm:mainsing}.

\begin{theorem}\label{thm:normal}
Let $C$ be a $GL(\bd^x_r)$-stable irreducible closed subvariety of $\rep_{A^x}(\bd^x_r)$, for some $0\leq r \leq \bd(x)$, and consider the irreducible $GL(\bd)$-stable variety $GL(\bd(x))\cdot C \subset \rep_A(\bd)$. Then $C$ is normal (resp. has rational singularities) if and only if  $GL(\bd(x))\cdot C$ is normal (resp. has rational singularities).
\end{theorem}

\begin{proof}
We apply Kempf's results \cite{Kempf76} on collapsing of vector bundles to our setup. Recall that $GL(\bd)$ acts on the affine space $\rep_Q(\bd)$, in which $\rep_{A^x}(\bd^x_r)$ is a closed $P_r$-stable subvariety, with the unipotent radical of $P_r$ acting trivially. In the language of \cite[Section~2]{Kempf76}, $P_r$ acts on $\rep_{A^x}(\bd^x_r)$ completely reducibly.
Hence, $P_r$ acts on $C$ completely reducibly as well. By Proposition \ref{prop:birational}, we have the birational ``collapsing map''
\[\Psi_C\colon GL(\bd(x))\times_{P_r} C \to GL(\bd(x)) \cdot C.\]
One implication on normality now follows from Proposition 1 of \cite{Kempf76}, and on rational singularities from Theorem 3 of \textit{ibid.}. The converse statements follow from the isomorphism (\ref{eq:invariant}) due to the direct summand property \cite{boutot}.
\end{proof}

Splitting successively as in Theorem \ref{thm:comprad}, we obtain the following statement. 

\begin{corollary}\label{cor:rankrat} Consider a radical square zero algebra $A=\kk Q/\kk Q_{\geq 2}$. For any $\br\leq \bd$, the variety $C_\br \subset \rep_A(\bd)$ has rational singularities.
\end{corollary}

The following example illustrates how to use the theorem to give other examples of algebras where all irreducible components of representation varieties have rational singularities.

\begin{example}\label{ex:singularity}
Consider again the algebra $A$ of Example \ref{ex:irrcomp}.  
We noted there that any representation variety for $A^x$ is the product of an affine space with a union of orbit closures in an equioriented type $\bbA$ quiver representation variety.
Each of these orbit closures, thus each irreducible component of any $\rep_{A^x}(\bd^x_r)$, is known to have rational singularities \cite{ADK81}.
Therefore every irreducible component of any $\rep_A(\bd)$ for this algebra has rational singularities by Theorem \ref{thm:normal}.
\end{example}

When splitting nodes of an algebra $A$ results in an algebra whose orbit closures are known to be normal or have rational singularities (e.g. Dynkin types $\bbA$ \cite{LM98, BZ01} or $\bbD$ \cite{BZ02}), then we can conclude the same for orbit closures of $A$.  See Examples \ref{ex:singularity2} and \ref{ex:singularity3} in the next section.

\subsection{Defining ideals}\label{sec:ideals}
In this section, we describe how the defining equations of irreducible varieties change under node splitting. For this, we start with a preliminary result for a particular class of radical square zero algebras.
We recall the notation used in the introduction before Theorem \ref{thm:maindef}, and the definition for $C_\br$ used in Theorem \ref{thm:comprad}.

\begin{proposition}\label{prop:loopideal}
Consider the algebra $A=\kk Q/\kk Q_{\geq 2}$, where $Q$ is the quiver
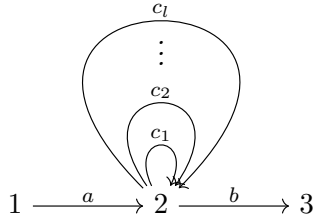
\[\begin{tikzcd}[column sep = large, outer sep= - 0.15ex]
1 \arrow[r, "a"] & 2
\arrow[
  out=115,
  in=65,
  loop,
  distance=0.8cm, "c_1"]
\arrow[
  out=125,
  in=55,
  loop,
  distance=1.8cm, "c_2"]
  \arrow[
  out=135,
  in=45,
  loop,
  distance=4.2cm, "c_l", "\vdots"']
  \arrow[r, "b"] & 3
\end{tikzcd}\]
with $l\in\mathbb{Z}_{\geq 0}$. Let $\bd=(m,d,n)$ be a dimension vector, $r$ an integer with $0\leq r \leq \min\{d, \, \lfloor \frac{m+d\cdot l}{l+1}\rfloor\}$, and consider the rank sequence $\br=\left(0,\,r,\,\min\{n,d-r\}\right)$. Then the prime ideal defining $C_\br$ is generated by the following polynomials in $\kk[\rep_{\kk Q}(\bd)]$:
\begin{enumerate}
\item The $(r+1)\times (r+1)$ minors of $H_2$;
\item The $(d-r+1) \times (d-r+1)$ minors of $T_2$;
\item The entries of \, $T_2 \cdot H_2$;
\item The trace of \,$X_{c_i}$, for $i=1,\dots, l$.
\end{enumerate}

\end{proposition}

\begin{proof}
We first note that the condition on $r$ means that $C_\br$ is a non-empty irreducible subvariety by Theorem \ref{thm:comprad}. Let us write $V_1 = \kk^m, \, W = \kk^d, \, V_2 = \kk^n$, so that we have a natural identification 
\[\rep_{\kk Q}(\bd) = V_1^* \oo W \, \oplus \,  W^* \oo V_2 \, \oplus \bigoplus_{i=1}^l \, W^* \oo W.\]
Denoting by $x$ the middle node $2$, we arrive by splitting $x$ to a resolution of singularities of $C_\br$ as in Proposition \ref{prop:birational}
\begin{equation}\label{eq:bundle}
GL(d) \times_{P_r} \rep_{\kk Q^x} (\bd^x_r) \,\, \longrightarrow \, C_\br.
\end{equation}
On $\Gr(r,d)$ we have the exact sequence of bundles 
\begin{equation}\label{eq:taut}
0\to \cR \to W \to \cQ \to 0,
\end{equation}
where $\cR$ (resp. $\cQ$) denotes the tautological subbundle (resp. quotient bundle) of rank $r$ (resp. $d-r$), and we write $W$ for the trivial bundle $\Gr(r,d) \times W$, for simplicity.
With this notation, the space $GL(d) \times_{P_r} \rep_{\kk Q^x} (\bd^x_r)$ in (\ref{eq:bundle}) corresponds to the bundle
\[ \cZ \, = \, V_1^* \oo \cR  \, \oplus \, \cQ^* \oo V_2  \, \oplus \bigoplus_{k=1}^l \cQ^* \oo \cR.\]
The idea now is to consider the intermediate bundle
\[\cY_1 \, = \, V_1^* \oo \cR  \, \oplus \, W^* \oo V_2  \, \oplus \bigoplus_{k=1}^l W^* \oo \cR.\]
Clearly, we have inclusions of bundles $\cZ  \subset \cY_1 \subset \Gr(r,d) \times \rep_{\kk Q}(\bd)$, and so a diagram
\[\xymatrix{
\cZ \ar[r]^{i} \ar[dr] & \cY_1 \ar[r]^{\hspace{-0.6in} i_1} \ar[d]^{p_1} & \Gr(r,d) \times \rep_{\kk Q}(\bd) \ar[dl]^{p} \ar[d]^q \\
  &   \Gr(r,d) & \rep_{\kk Q}(\bd)
}\]
Let $\xi$ (resp. $\xi_1$) be the dual of the quotient bundle $\cY_1/\cZ$ (resp. $(\Gr(r,d)\times\rep_{\kk Q}(\bd))/\cY_1$) on $\Gr(r,d)$, namely
\[ \xi= \cR \oo V_2^* \oplus \bigoplus_{k=1}^l \cR\oo \cR^*, \quad \mbox{and } \, \xi_1= V_1 \oo \cQ^* \oplus \bigoplus_{k=1}^l W\oo \cQ^*.\]
We have the following Koszul complex on $\cY_1$ (cf. \cite[Proposition 5.1.1]{weymanbook})
\[ K_\bullet \colon \, 0 \to \!\!\bigwedge^{r(n+rl)} \!\! p_1^*(\xi)  \to \dots  \to \bigwedge^i p_1^*(\xi)  \to \dots \to p_1^*(\xi) \to \cO_{\cY_1}\to i_* \cO_{\cZ} \to 0.\]
Put $S=\kk[\rep_{\kk Q}(\bd)]$ and denote by $J_1\subset S$ the ideal generated by the equations of type (1), i.e. $(r+1)\times (r+1)$ minors of $H_2$. Note that $\cY_1$ is the desingularization of $\op{Spec}(S/J_1)$ via $q \circ i_1$ (cf. \cite[Section 6.1]{weymanbook}). Let $I$ denote the defining ideal of $C_\br$. Since $\op{Spec}(S/J_1)$ (resp. $C_\br$) has rational singularities, we have $\bR^\bullet q_* i_{1*} \cO_{\cY_1} \cong S/J_1$ (resp. $\bR^\bullet q_* i_{1*} i_* \cO_{\cZ} \cong S/I$). 

Now we analyze the other terms of the complex $K_\bullet$. We claim that we have
\begin{equation}\label{eq:claim}
H^k(\cY_1, \bigwedge^t p_1^*(\xi)) = 0, \mbox{ for all } k \geq t-1\geq 0 \mbox{ unless } (k,t) \in \{ (0,1) ,\, (d-r, d-r+1) \}.
\end{equation}
To show the claim, we consider the natural isomorphisms (cf. \cite[Proposition 5.1.1(b)]{weymanbook})
\[ \bR^k q_*i_{1*} (\bigwedge^t p_1^*(\xi)) \cong  H^k(\cY_1, \bigwedge^t p_1^*(\xi)) \cong H^k(\Gr(r,d), \Sym \eta \oo \bigwedge^t \xi), \]
where $\eta= \cY_1^*$. Through the Cauchy formulas \cite[Corollary 2.3.3]{weymanbook}, $\Sym \eta \oo \bigwedge^t \xi$ can be decomposed as a direct sum of bundles of the form
\begin{equation}\label{eq:summand}
 \SSS_\zl V_1 \oo \SSS_\mu W \oo \SSS_\nu V_2^* \oo \SSS_\za \cR \oo \SSS_\zb \cR^*, \,\mbox{for some partitions } \zl,\mu,\nu,\za,\zb \mbox{ with } |\za| = t.
 \end{equation}
Here $|\za|$ denotes the size of the partition $\za$, and $\SSS_\za (-)$ denotes the Schur functor (see \cite[Section 2]{weymanbook}). Hence, for the vanishing in (\ref{eq:claim}), it is enough to show the corresponding vanishing of $H^k(\Gr(r,d), \SSS_\za \cR \oo \SSS_\zb \cR^*)$, with $|\za|=t$. 

Assume that $H^k(\Gr(r,d), \SSS_\za \cR \oo \SSS_\zb \cR^*)\neq 0$, for some $k>0$ and $|\za|\leq k+1$. Then by the Littlewood--Richardson rule and Bott's theorem (see \cite[Theorem 2.3.4 and Corollary 4.1.9]{weymanbook}), we see by inspection that we must have $\za=(d-r+1)$, so that $t=d-r+1$ and $k=d-r$. This proves the claim (\ref{eq:claim}).

Let $f$ denote the map $f: H^0(\cY_1, p_1^*(\xi)) \to S/J_1$ induced from the beginning of $K_\bullet$. Pushing forward the complex $K_\bullet$ via $q\circ i_{1}$, we get using (\ref{eq:claim}) via a standard spectral sequence argument an exact sequence of $S$-modules
\[ H^{d-r}(\cY_1, \! \bigwedge^{d-r+1} \!\! p_1^*(\xi))  \xrightarrow{\,\,\,g\,\,\,} (S/J_1)/(\op{im} f) \to S/I \to 0.\]
To finish the proof, we show that $H^0(\cY_1, p_1^*(\xi))$ contributes with the equations of type (3), (4) via $f$, while $H^{d-r}(\cY_1, \bigwedge^{d-r+1} p_1^*(\xi))$ contributes with the equations of type (2) via $g$.

By \cite[Basic Theorem 5.1.2]{weymanbook}, we have a surjective map of $S$-modules
\[\bigoplus_{k\geq 0} H^k(\Gr(r,d), \xi \oo \bigwedge^k \xi_1) \oo S(-k) \, \longrightarrow \, H^0(\cY_1, p_1^*(\xi)).\]
Using Bott's theorem as before, we see that $H^k(\Gr(r,d), \xi \oo \bigwedge^k \xi_1)\neq 0$ if and only if $k\in \{0,1\}$. When $k=0$ we get a $GL(d)$-invariant element in degree $1$ for each loop, which must be sent to the equations of type (4) via $f$. When $k=1$, we obtain a $GL(d)$-representation in degree $2$ for each pair of composable arrows, which must be sent to the equations of type (3) via $g$.

Now we analyize $H^{d-r}(\cY_1, \bigwedge^{d-r+1} p_1^*(\xi))$ by introducing another intermediate bundle 
\[\cY_2 = V_1^* \oo W \oplus \cQ^*\oo V_2 \oplus \bigoplus_{k=1}^l \cQ^* \oo W.\]
Again, we have inclusions of bundles $\cZ  \xrightarrow{j} \cY_2 \xrightarrow{i_2} \rep_{\kk Q}(\bd)$, and $\cY_2$ is a desingularization of $\op{Spec} (S/J_2)$ via $q \circ i_2$, where $J_2$ is the ideal generated by the minors of type (2). Let $\xi_2$ denote the dual bundle of the quotient $\rep_{\kk Q}(\bd)/\cY_2$, that is
\[\xi_2 = \cR \oo V_2^* \oplus \bigoplus_{k=1}^l \cR \oo W^*.\]
We have a Koszul complex on $\Gr(r,d)\times \rep_{\kk Q}(\bd)$ of the form 
\[L_\bullet \colon \, 0 \to \! \bigwedge^{r(l+n)} \! p^*(\xi_2)  \to \dots  \to \bigwedge^i p^*(\xi_2)  \to \dots \to p^*(\xi_2) \to \cO_{\Gr(r,d)\times \rep_{\kk Q}(\bd)}\to i_{2*} \cO_{\cY_2} \to 0.\]
The inclusions of bundles $i,j,i_1,i_2$ induce a map of complexes $L_\bullet \to i_{1*} K_\bullet$. The claim analogous to (\ref{eq:claim}) holds for $L_\bullet$ as well. Note that $H^0(\Gr(r,d), \xi_2)=0$. Hence, the map of complexes induces via spectral sequence a commutative diagram
\[\xymatrix{
H^{d-r}(\Gr(r,d)\times \rep_{\kk Q}(\bd), \!\! \displaystyle\bigwedge^{d-r+1} \!\! p^*(\xi_2))  \ar[r] \ar[d]^h & S \ar[r] \ar[d] & S/J_2 \ar[d] \ar[r] & 0 \\
H^{d-r}(\cY_1, \displaystyle\bigwedge^{d-r+1} p_1^*(\xi))  \ar[r] &  (S/J_1)/(\op{im} f) \ar[r] & S/I \ar[r] & 0
}\]
with exact rows. As $J_2$ is generated by the equations of type (2), it is enough to show that $h$ is surjective. We have seen by decomposing into summands that $H^{d-r}(\cY_1, \bigwedge^{d-r+1} p_1^*(\xi))$ yields non-zero cohomology only for summands as in (\ref{eq:summand}) with $\za = (d-r+1)$. By decomposing in an analogous way, we see that the contributing summands to $H^{d-r}(\Gr(r,d)\times \rep_{\kk Q}(\bd),  \bigwedge^{d-r+1} p^*(\xi_2))$ are of the form 
\begin{equation}\label{eq:summand2}
 \SSS_\zl V_1 \oo \SSS_\mu W \oo \SSS_\nu V_2^* \oo \SSS_{(d-r+1)} \cR \oo \SSS_\zb W^*, \,\,\mbox{for some partitions } \zl,\mu,\nu,\zb.
 \end{equation}
The map $h$ descends to a map on $(d-r)$th cohomology from the summands of type  (\ref{eq:summand2}) to their corresponding summands of type (\ref{eq:summand}) (with $\za=(d-r+1)$). To show surjectivity of $h$, we are left to show that the map $ \SSS_{(d-r+1)} \cR \oo \SSS_{\zb} W^* \to  \SSS_{(d-r+1)} \cR \oo \SSS_{\zb} \cR^*$ induces a surjective map on $(d-r)$th cohomology. By Bott's theorem we have
\[H^{d-r}(\Gr(d,n), \, \SSS_{(d-r+1)} \cR \oo \SSS_{\zb} \cR^* ) = \SSS_{(1^{d-r+1}, -\beta)} W,\] 
where the latter term is zero if $\beta$ has $r$ parts. Consider the exact sequence of bundles
\[ 0 \to \SSS_{d-r+1} \cR \, \oo \cK \to \SSS_{(d-r+1)} \cR \oo \SSS_{\zb} W^*  \to \SSS_{(d-r+1)} \cR \oo \SSS_{\zb} \cR^* \to 0.\]
We show that $H^{d-r+1}(\Gr(d,k), \SSS_{(d-r+1)}\cR \, \oo \cK)$ has no summand $\SSS_{(1^{d-r+1}, -\beta)}W$, thus proving the claim on surjectivity. The tautological sequence (\ref{eq:taut}) gives rise to a filtration of $\SSS_\beta W^*$ with composition factors that are direct summands of $\SSS_{\beta} (\cQ^* \oplus \cR^*)$. Hence, by the corresponding Littlewood--Richardson rule (see \cite[Proposition 2.3.1]{weymanbook}) $\cK$ has a filtration with associated graded a direct sum of terms $\SSS_{\beta_1} \cQ^* \oo  \SSS_{\beta_2} \cR^*$, where $\beta_1,\beta_2$ are partitions with $\beta_1 \neq 0$. By Bott's theorem again, we see by inspection that since $\beta_1\neq 0$, the cohomology groups of the bundle $\SSS_{\beta_1} \cQ^* \oo \SSS_{(d-r+1)} \cR \oo  \SSS_{\beta_2} \cR^*$ cannot have representations corresponding to partitions with strictly positive $(d-r+1)$th part. Since the cohomology of the associated graded has no summand $\SSS_{(1^{d-r+1}, -\beta)}W$, neither does $H^{d-r+1}(\Gr(d,k), \SSS_{(d-r+1)} \cR \, \oo \cK)$. Thus $h$ is onto, finishing the proof of the lemma.
\end{proof}

\begin{remark}
The equations given in Proposition \ref{prop:loopideal} do not form a minimal set of generators in general. In the case $l=0$, the equations are minimal except some degenerate cases (cf. \cite[Theorem 4.2]{lor2}). But in the case $l>0$, if $2r\leq d$ (resp. if $2r>d$) then the minors of $T_2$ (resp. $H_2$) that don't involve rows from $X_b$ (resp. columns from $X_a$) can be obtained by Laplace expansions from the minors of $H_2$ (resp. $T_2$). When the quiver is one loop, minimal equations are given in \cite[Example 5.6]{WeyNilp}.
\end{remark}

Let $Q$ be a quiver and $\bd$ a dimension vector. Let $C$ be a $GL(\bd^x_r)$-stable closed subvariety of $\rep_{\kk Q^x}(\bd^x_r)$ (for some $0\leq r \leq \bd(x)$). Denote by $U^-$ the transpose of the unipotent radical of $P_r$. We have a natural map $\kk[GL(\bd(x)) \cdot C] \to \kk[C]$, that descends onto an isomorphism (e.g. see \cite[Proposition 3.14]{collapsechar})
\begin{equation}\label{eq:invariant}
\kk[GL(\bd(x)) \cdot C]^{U^-} \, \xrightarrow{\,\,\cong\,\,} \, \kk[C].
\end{equation}
As a consequence, the algebra $\kk[C]$ is a direct summand of $\kk[GL(\bd(x)) \cdot C]$ as a $\kk[C]$-module. In particular, when $C=\rep_{\kk Q^x}(\bd^x_r)$ we can identify polynomials in $\kk[\rep_{\kk Q^x}(\bd^x_r)]$ naturally inside $\kk[\rep_{\kk Q}(\bd)]$. With this, we formulate the following result.

\begin{theorem}\label{thm:defideal}
Let $C$ be a $GL(\bd^x_r)$-stable closed subvariety of $\rep_{\kk Q^x}(\bd^x_r)$, for some $0\leq r \leq \bd(x)$. Let $\cP$ be a set of polynomials in $\kk[\rep_{\kk Q^x}(\bd^x_r)]$ that generate the radical ideal defining $C$. Then the radical ideal defining $GL(\bd(x))\cdot C$ is generated by the following polynomials in $\kk[\rep_{\kk Q}(\bd)]$:
\begin{enumerate}
\item The $(r+1)\times (r+1)$ minors of $H_x$;
\item The $(\bd(x)-r+1) \times (\bd(x)-r+1)$ minors of $T_x$;
\item The entries of \, $T_x \cdot H_x$;
\item The trace of \,$X_{\gamma}$, for every loop $\gamma\in Q_1$ at $x$.
\item A basis of $\,\op{span}_{\kk} \, GL(\bd(x)) \cdot \cP$.
\end{enumerate}
\end{theorem}

\begin{proof}
For $C=\rep_{\kk Q^x}(\bd^x_r)$ the statement follows from Proposition \ref{prop:loopideal}. The relative result is now a consequence of \cite[Theorem 3.14]{collapsechar}.
\end{proof}

As saturation of minors giving the required rank conditions are again minors of such type, we obtain the following result.

\begin{corollary}\label{cor:radef}
Consider a radical square zero algebra $A=\kk Q/\kk Q_{\geq 2}$ and let $C_\br \subset \rep_A(\bd)$ be non-empty. Then the prime ideal of $C_\br$ is generated by the following polynomials in $\kk[\rep_{\kk Q}(\bd)]$, as $x$ runs through all the vertices in $Q_0$:
\begin{enumerate}
\item The $(\br(x)+1)\times (\br(x)+1)$ minors of $H_x$;
\item The $(\bd(x)-\br(x)+1) \times (\bd(x)-\br(x)+1)$ minors of $T_x$;
\item The entries of \, $T_x \cdot H_x$;
\item The trace of \,$X_{\gamma}$, for every loop $\gamma\in Q_1$ at $x$.
\end{enumerate}
\end{corollary}

The following examples illustrate how to use these results.

\begin{example}\label{ex:singularity2}
Consider the radical square zero algebra $A=\kk Q/\kk Q_{\geq 2}$, where
\vspace{2ex}
\[Q= \quad \xymatrix{
\bullet \ar[r] & \bullet \ar[r] \ar@(ul,ur) & \bullet \ar[r] \ar@(ul,ur)& \dots \ar[r] & \bullet \ar[r] \ar@(ul,ur)& \bullet
}.\]
Splitting the nodes we arrive at the hereditary algebra $A^{\op{sp}}$ of  the type $\bbA$ quiver
\[Q^{\op{sp}} = \quad \xymatrix{
\bullet \ar[r] & \bullet & \bullet \ar[r] \ar[l]& \bullet & \bullet \ar[r] \ar[l]& \dots & \bullet \ar[l] \ar[r] & \bullet 
}.\]
Orbit closures of representations for type $\bbA$ quivers have rational singularities by \cite{BZ01}, 
so combining Corollary \ref{cor:orbclos} with Theorem \ref{thm:normal} shows that all orbit closures for $A$ have rational singularities as well.
Furthermore, defining equations for orbit closures of type $\bbA$ quivers were described  in \cite[Theorem 6.4]{RZ13} (see also \cite[Section~3.1]{KR}), so Theorem \ref{thm:defideal} 
gives defining equations for orbit closures of representations of $A$.  
Note that the algebra $A$ is of finite representation type (see Corollary \ref{cor:indec}), so the irreducible components of its representation varieties are always orbit closures.

We illustrate the defining equations in more detail for a specific size $Q$, which readily generalizes.  Take the dimension vector below and let $A_i, B_i$ be generic matrices of variables of appropriate sizes.
\[Q= \quad \xymatrix{
d_1 \ar[r]^{A_1} & d_2 \ar[r]^{A_2} \ar@(ul,ur)^{B_1} & d_3 \ar[r]^{A_3} \ar@(ul,ur)^{B_2} & d_4
}\]
\[Q^{\op{sp}} = \quad \xymatrix{
d_1 \ar[r]^{A'_1} & d_2' & d_2'' \ar[r]^{A'_2} \ar[l]_{B'_1}& d_3' & d_3'' \ar[r]^{A'_3} \ar[l]_{B'_2}& d_4
}\]
Let $\ol{O} \subset \rep_A(\bd)$ be an orbit closure and $\ol{O}^{\op{sp}} \subset \rep_{Q^{\op{sp}}}(\bd^{\op{sp}})$ the associated type $\bbA$ orbit closure.  We follow the numbering of Theorem \ref{thm:defideal}, noting that all specific minor sizes below are determined by $\ol{O}$. Polynomials of type (1) and (2) are minors of the matrices
\[
H_2=\begin{bmatrix}A_1 & B_1 \end{bmatrix} \quad
H_3=\begin{bmatrix}A_2 & B_2 \end{bmatrix} \quad
T_2=\begin{bmatrix}B_1 \\ A_2 \end{bmatrix} \quad
T_3=\begin{bmatrix}B_2 \\ A_3 \end{bmatrix}.
\]
Polynomials of type (3) are the entries of $A_2 A_1, B_1 A_1, B_1^2, A_2 B_1$ and $A_3 A_2, B_2 A_2, B_2^2, A_3 B_2$, and type (4) are the traces of $B_1$ and $B_2$.  Finally, a set of polynomials $\cP$ defining $\ol{O}^{\op{sp}}$ comes from minors that give rank conditions on the 15 submatrices of
\[
\begin{bmatrix}
A_1' & B_1' & 0\\
0 & A_2' & B_2'\\
0 & 0 & A_3'
\end{bmatrix},
\]
which correspond to the 15 connected subquivers of $Q^{\op{sp}}$ that contain at least one arrow.
Thus, equations of type (5) come from a basis of the linear span of the $GL(d_2)\times GL(d_3)$-saturation of $\cP$, which can be chosen to be the respective minors (giving the same rank conditions) of the matrix
\[
\begin{bmatrix}
A_1 & B_1 & 0\\
0 & A_2 & B_2\\
0 & 0 & A_3
\end{bmatrix}.
\]
\end{example}

\begin{example}\label{ex:singularity3}
Consider the following algebra $A=\kk Q/I$ obtained by deleting vertex 6 from the algebra in Example \ref{ex:singularity}, so $I$ is generated by the 6 relations necessary to make $x$ a node, along with the relation $cba=0$.
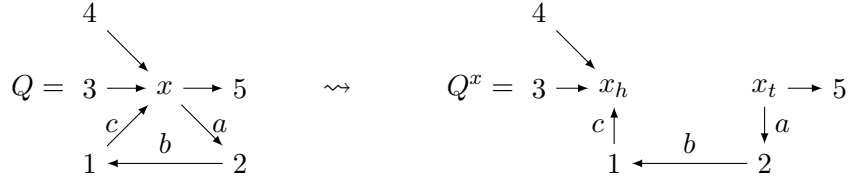
\[
Q=\vcenter{\hbox{\begin{tikzpicture}[point/.style={shape=circle,fill=black,scale=.5pt,outer sep=3pt},>=latex]
  \node (x) at (0,1) {$x$};
  \node (1) at (-1,0) {1};
  \node (2) at (1,0) {2};
  \node (3) at (-1,1) {3};
  \node (4) at (-1,2) {4};
  \node (5) at (1,1) {5};
  \path[->]
	(4) edge (x)
	(3) edge (x) 
	(1) edge node[left] {$c$} (x)
	(2) edge node[above] {$b$} (1)
	(x) edge (5)
	(x) edge node[right] {$a$} (2);
   \end{tikzpicture}}}
\qquad\rightsquigarrow \qquad
\quad
Q^x=\vcenter{\hbox{\begin{tikzpicture}[point/.style={shape=circle,fill=black,scale=.5pt,outer sep=3pt},>=latex]
  \node (xh) at (-1,1) {$x_h$};
  \node (xt) at (1,1) {$x_t$};
  \node (1) at (-1,0) {1};
  \node (2) at (1,0) {2};
  \node (3) at (-2,1) {3};
  \node (4) at (-2,2) {4};
  \node (5) at (2,1) {5};
  \path[->]
	(4) edge (xh)
	(3) edge (xh) 
	(1) edge node[left] {$c$} (xh)
	(2) edge node[above] {$b$} (1)
	(xt) edge (5)
	(xt) edge node[right] {$a$} (2);
   \end{tikzpicture}}}
\]
Orbit closures of $A^x$ are orbit closures for a type $\bbD$ quiver, and thus have rational singularities by \cite{BZ02} (see also \cite{KR19} for a proof by entirely different methods).  
Therefore, combining Corollary \ref{cor:orbclos} with Theorem \ref{thm:normal} shows that all orbit closures for $A$ have rational singularities. 

Since defining equations for orbit closures are not known in type $\bbD$, we cannot yet get defining equations for orbit closures for $A$.  However, the irreducible components of representation varieties for $A^x$ are defined by a relation which is constrained to a type $\bbA$ subquiver, so defining equations for irreducible components of representation varieties of $A$ are obtained from Theorem \ref{thm:defideal} as in the previous example.
Note that the algebra $A$ is of finite representation type (see Corollary \ref{cor:indec}).
\end{example}

\begin{remark}\label{rem:char}
Many of our results carry over to a field of arbitrary characteristic. In fact, the statements of Corollary \ref{cor:rankrat} and Corollary \ref{cor:radef} (the latter in the case when no loops are present) are true over any field. This relies on a characteristic-free extension of Kempf's work \cite{Kempf76} and further results on defining ideals, which is pursued in \cite{collapsechar}. In the case when there are loops, more equations are needed besides the ones in Corollary \ref{cor:radef} as can be seen already for the one loop quiver \cite{Strick2}.
\end{remark}


\section{Moduli spaces of representations}\label{sec:moduli}
In this section we apply the results above to moduli spaces of semistable representations. We give only a minimal recollection of the background here, referring the reader to A. D. King's original paper \cite{King94} or \cite{Reineke08, DWbook} for more detailed treatment.
We then make some observations about semistable representations of algebras with nodes, then provide an example where Theorem \ref{thm:mainsing} applies to study structure of moduli spaces of representations.
We continue to assume $\charac \kk = 0$.

\subsection{Background and notation}
The idea of King was to apply the general machinery of Geometric Invariant Theory (GIT) \cite{MFbook,Newstead09} to study representations of finitely generated algebras.  The tools of invariant theory are very useful for understanding closed orbits of the action of a reductive group on a variety.  However, in the situation of $GL(\bd)$ acting on $\rep_A(\bd)$, the closed orbits correspond to just the semisimple representations, so there is only one such representation per $\bd$ when $A$ is finite-dimensional.

It turns out that there are many subcategories of the category of representations of $A$ with richer collections of semisimple objects.  From the representation theory perspective, each choice of weight $\theta \in \ZZ Q_0$ determines an abelian subcategory of $\theta$-semistable representations of $A$.
The simple objects of this category are called $\theta$-stable representations.  The choice of $\theta$ can be arbitrary in our results below. 

More precisely, for each $\bd$ satisfying $\theta \cdot \bd=0$, the collection of $\theta$-semistable points of $\rep_A(\bd)$ is defined by
\[
\rep_A(\bd)^{ss}_\theta:=\setst{M \in \rep_A(\bd)}{\forall N \leq M,\ \theta\cdot \underline{\dim} N \leq 0}.
\]
This is an open subvariety of $\rep_A(\bd)$ (possibly empty!).
There is a corresponding projective variety $\cM(\bd)^{ss}_\theta$ known as the \emph{moduli space of $\theta$-semistable representations of $A$ of dimension vector $\bd$}, and morphism of varieties
\[
\pi\colon \rep_A(\bd)^{ss}_\theta \onto \cM(\bd)^{ss}_\theta
\]
which is a quotient map in a sense made precise by GIT.
The $\theta$-stable points (those $\theta$-semistable $M$ such that $\theta \cdot N < 0$ for all proper, non-zero $N < M$) form an open subvariety $\rep_A(\bd)^{s}_\theta\subset \rep_A(\bd)^{ss}_\theta$ on which $GL(\bd)$ acts freely; $\pi$ is an honest quotient map when restricted to this subvariety (again possibly empty).
We extend the notations above to subsets $C \subset \rep_A(\bd)$, writing $C_{\theta}^{ss}$ for the set of $\theta$-semistable points in $C$, and $\cM(C)^{ss}_\theta$ for the image of $\pi(C_{\theta}^{ss})$.  We say \emph{$C$ is $\theta$-semistable} if $C_{\theta}^{ss} \neq \emptyset$.

Since the $\theta$-stable representations are the simple objects in the abelian category of $\theta$-semistable representations, every $\theta$-semistable representation $M$ has a well-defined set of $\theta$-stable composition factors from the Jordan-H\"older theorem, and associated graded representation $\gr_\theta(M)$.

The process of passing from $M$ to $\gr_\theta(M)$ can be carried out in a geometric setting, known as a $\theta$-stable decomposition.  We follow the exposition of \cite[\S2.4]{CKdecomp} which is a slight generalization of \cite[Section 3C]{Chindris10}, based on the original idea of \cite{DW11} in the case that $A=\kk Q$ for an acyclic quiver $Q$.

\begin{definition}\label{def:thetastable}
Let $C$ be a $GL(\bd)$-invariant, irreducible, closed subvariety of $\rep_A(\bd)$, and assume $C$ has a non-empty subset of $\theta$-semistable points. Consider a collection $\{C_i \subset \rep_A(\bd_i)\}_{i=1}^k$ of irreducible components such that each has a non-empty subset of $\theta$-stable points, $C_i \neq C_j$ for $i \neq j$, and also consider some multiplicities $m_i \in \ZZ_{>0}$, for $i=1, \dots, k$.
We say that $\{(C_i, m_i)\}_{i=1}^k$ is a \emph{$\theta$-stable decomposition of $C$} if, for a general representation $M \in C^{ss}_{\theta}$, its corresponding $\gr_\theta(M)$ is in $C_1^{\oplus m_1} \oplus \cdots \oplus C_k^{\oplus m_k} $, and write
\[
C=m_1C_1\pp \ldots \pp m_k C_k. \qedhere
\]
\end{definition}

It is shown in \cite[Prop.~3]{CKdecomp} that any $GL(\bd)$-invariant, irreducible, closed subvariety of $C\subseteq \rep_A(\bd)$ such that $C^{ss}_\theta \neq \emptyset$ admits a $\theta$-stable decomposition.
The following result makes precise how the geometry of a moduli space of $\theta$-semistable representations is constrained (and in some cases completely determined) by the geometry of moduli spaces arising from its $\theta$-stable decomposition.  
Here, the \emph{$m^{th}$ symmetric power} $S^m(X)$ of a variety $X$ is the quotient of $\prod_{i=1}^m X$ by the action of the symmetric group on $m$ elements which permutes the coordinates.

\begin{theorem}\label{thm:CKmain}\cite{CKdecomp} Let $A$ be a finite-dimensional algebra and let $C \subset \rep_A(\bd)^{ss}_\theta$ be an irreducible component such that $C^{ss}_\theta \neq \emptyset$.
Let $C=m_1C_1\pp \ldots \pp m_k C_k$ be a $\theta$-stable decomposition of $C$ where $C_i \subset \rep_A(\bd_i)$, $1 \leq i \leq k$, are pairwise distinct $\theta$-stable irreducible components.

If $\cM(C)^{ss}_{\theta}$ is an irreducible component of $\cM(\bd)^{ss}_{\theta}$, then there is a natural morphism
\[
\Psi\colon  S^{m_1}(\cM(C_1)^{ss}_{\theta}) \times \ldots \times S^{m_r}(\cM(C_k)^{ss}_{\theta})  \to \cM(C)^{ss}_{\theta}
\]
which is finite, and birational. In particular, if $\cM(C)^{ss}_{\theta}$ is normal then $\Psi$ is an isomorphism. 
\end{theorem}

Every irreducible component of $\cM(\bd)^{ss}_{\theta}$ is of the form $\cM(C)^{ss}_{\theta}$ where $C$ is an irreducible component of $\rep_A(\bd)$, so this covers all of them.
Here we have combined the three parts of the main theorem of \cite{CKdecomp} for simplicity; this is enough for our application.
We also note that the map of this theorem is quite simplistic on the set-theoretical level, sending a list of representations to their direct sum.  The entire content is that $\Psi$ is a morphism of varieties with nice properties.

 \subsection{Moduli spaces of representations of algebras with nodes}
In this subsection we observe that semistability of representations for an algebra $A$ with a node $x$ is particularly simply-behaved around $x$.
Recall $h_x, t_x$ from \eqref{eq:hxtx}.

\begin{proposition}\label{prop:nodestable}
Assume that $A=\kk Q/I$ with $x\in Q_0$ a node, and let $\theta \in \ZZ Q_0$ be a weight. Consider a $\theta$-semistable representation $M$ of $A$. Then one of the following occurs:
\begin{itemize}
\item[(a)] If $\theta(x)<0$, then $h_x(M)$ is surjective and $t_x(M)=0$;
\item[(b)] If $\theta(x)>0$, then $h_x(M)=0$ and $t_x(M)$ is injective;
\item[(c)] If $\theta(x)=0$, then the $\theta$-semistability of $M$ is equivalent to the $\theta$-semistability of $M'$, where $M'$ is obtained from $M$ by putting $M'_x=0$ (so $h_x(M')=t_x(M')=0$) and leaving the rest of the maps of $M$ unchanged.
\end{itemize}
\end{proposition}

\begin{proof}
Clearly, we can assume $M_x\neq 0$. Consider first the case when $M$ is $\theta$-stable. In particular, it then must be a Schur representation. By Corollary \ref{cor:schur}, either $h_x(M)$ or $t_x(M)$ is zero. Assume that the latter holds (the former case is analogous). Then $M_x\neq 0$ implies that $S_x$ is a subrepresentation of $M_x$. Since $M$ is $\theta$-stable, either $S_x=M$ in which case $\theta(x) = \theta \cdot \underline{\dim} S_x =0$, or $\theta(x)<0$ in which case $h_x(M)$ is onto as $M \neq S_x$ is indecomposable.

Now let $M$ be $\theta$-semistable. Let us prove part (a) (part (b) is analogous). Then $\theta(x)<0$ implies that for all $\theta$-stable composition factors $N$ of $M$ the map $h_x(N)$ is onto. Hence, $h_x(\gr_\theta(M))$ is onto, and then $h_x(M)$ is onto as well. 

We are left with part (c), in which case the simple $S_x$ is $\theta$-stable.  Then (c) follows from the fact that the set of $\theta$-semistable representations forms an abelian category which is closed under extensions, and that $x$ being a node forces every copy of $S_x$ to lie in the top or socle of any representation.
\end{proof}

We now spell out the implication of Proposition \ref{prop:nodestable} for moduli spaces of algebras with nodes.
Note that the results for moduli spaces are not trivial consequences of the embedding of categories resulting from node splitting described in Section \ref{sec:nodesplit}, since having a bijection between the points of two varieties does not guarantee they are isomorphic.
Below, we add subscripts to the moduli space notation to clarify which algebra is being considered.

\begin{proposition}\label{prop:nodemoduli}
Assume that $A=\kk Q/I$ with $x\in Q_0$ a node, and let $\theta \in \ZZ Q_0$ be a weight.
Consider a representation variety $\rep_A(\bd)$ that is $\theta$-semistable.
Then there exists an algebra $A'$, which is a proper quotient of $A$ by an ideal generated by some arrows from $Q$ making $x$ a sink or source in $A'$, such that $\cM_A(\bd)^{ss}_{\theta}=\cM_{A'}(\bd)^{ss}_{\theta}$.
\end{proposition}

\begin{proof}
If $\theta(x)<0$, case (a) of Proposition \ref{prop:nodestable} shows that we can obtain such $A'$ from $A$ as the quotient by the ideal generated by all arrows with tail $x$, since
$\rep_{A'}(\bd)^{ss}_\theta = \rep_A(\bd)^{ss}_\theta$ under the identification $\rep_{A'}(\bd)\subseteq \rep_A(\bd)$.
The case of $\theta(x) >0$ follows from (b) similarly.

If $\theta(x)=0$, consider the algebra $A'$ obtained from $A$ as the quotient by the ideal generated by all arrows incident to $x$.  The closed embedding $\rep_{A'}(\bd)\subseteq \rep_A(\bd)$ immediately gives $\cM_{A'}(\bd)^{ss}_{\theta} \subseteq \cM_A(\bd)^{ss}_{\theta}$.
On the other hand, the proof of case (c) of Proposition \ref{prop:nodestable} shows that for $M \in \rep_A(\bd)^{ss}_\theta$ we have $\gr_\theta(M) = M' \oplus S$, where $S$ is a direct sum of copies of $S_x$ and $M'$ is not supported at $x$,
which is to say that $\gr_\theta(M) \in \rep_{A'}(\bd)^{ss}_\theta$.  This gives the reverse inclusion 
$\cM_A(\bd)^{ss}_{\theta}\subseteq \cM_{A'}(\bd)^{ss}_{\theta}$.
\end{proof}

We illustrate this for radical square zero algebras in the following example.

\begin{example}\label{ex:rad2moduli}
Let $Q$ be any quiver, $A=\kk Q / \kk Q_{\geq 2}$, and $\theta$ a weight for $Q$ which is not identically zero.  Define a quiver $Q^\theta$ by:
\begin{itemize}
\item deleting vertices $x$ and all incident arrows to $x$ if $\theta(x)=0$;
\item deleting all arrows with head $x$ if $\theta(x)>0$;
\item deleting all arrows with tail $x$ if $\theta(x)<0$.
\end{itemize}
Then $Q^\theta$ is an acyclic (in fact, bipartite) quiver and 
Proposition \ref{prop:nodemoduli} implies that $\cM_A(\bd)^{ss}_\theta = \cM_{\kk Q^\theta}(\bd)^{ss}_\theta$ for any dimension vector $\bd$.  Thus any moduli space for a radical square zero algebra is equal to a moduli space for some bipartite quiver without relations.
\end{example}

\subsection{Applications of main results to moduli spaces of representations}
Although semistability around a node has a simple behavior, we can apply our main results to moduli spaces of representations in more interesting situations as well. 

If we add to the quiver of a radical square zero algebra some additional arrows and vertices (without adding additional relations) then the irreducible components of representation varieties of the obtained algebra are still normal, hence the map $\Psi$ in Theorem \ref{thm:CKmain} is again an isomorphism. On the other hand, such algebras have richer moduli spaces. We illustrate these considerations with the following example.

\begin{example}\label{ex:moduli}
Consider a quiver and dimension vector of the form 
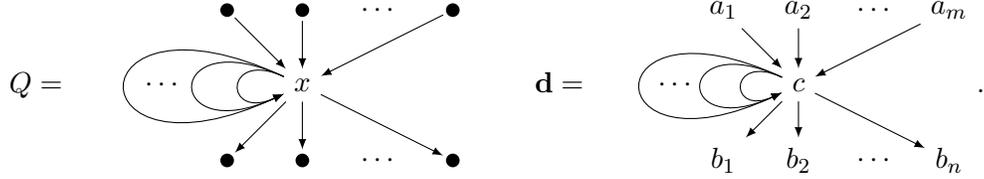
\[
Q=
\vcenter{\hbox{\begin{tikzpicture}[point/.style={shape=circle,fill=black,scale=.5pt,outer sep=3pt},>=latex]
   \node[point] (y1) at (-2,1) {};
  \node[point] (y2) at (-1,1) {};
  \node[outer sep=-2pt] at (0,1) {${\cdots}$};
  \node[point] (yr) at (1,1) {};
   \node[label={[label distance=0.5cm]180:$\cdots$}] (x) at (-1,0) {${x}$};
   \node[point] (z1) at (-2,-1) {};
  \node[point] (z2) at (-1,-1) {};
  \node[outer sep=-2pt] at (0,-1) {${\cdots}$};
  \node[point] (zs) at (1,-1) {};
\draw[->] (x) to [out=155,in=205,looseness=10] (x);
\draw[->] (x) to [out=150,in=210,looseness=20] (x);
  \path[->]
	(y1) edge  (x)
	(y2) edge  (x)
	(yr) edge  (x)
  	(x) edge  (z1)
	(x) edge  (z2)
	(x) edge  (zs);
   \end{tikzpicture}}}
\hspace{1cm}
   \bd=
\vcenter{\hbox{\begin{tikzpicture}[point/.style={shape=circle,fill=black,scale=.5pt,outer sep=3pt},>=latex]
   \node (y1) at (-2,1) {$a_1$};
  \node (y2) at (-1,1) {$a_2$};
  \node[outer sep=-2pt] at (0,1) {${\cdots}$};
  \node (yr) at (1,1) {$a_m$};
   \node[label={[label distance=0.5cm]180:$\cdots$}] (x) at (-1,0) {${c}$};
\draw[->] (x) to [out=155,in=205,looseness=10] (x);
\draw[->] (x) to [out=150,in=210,looseness=20] (x);
   \node (z1) at (-2,-1) {$b_1$};
  \node (z2) at (-1,-1) {$b_2$};
  \node[outer sep=-2pt] at (0,-1) {${\cdots}$};
  \node (zs) at (1,-1) {$b_n$};
  \path[->]
	(y1) edge  (x)
	(y2) edge  (x)
	(yr) edge  (x)
  	(x) edge  (z1)
	(x) edge  (z2)
	(x) edge  (zs);
   \end{tikzpicture}}}.
\]
To study its moduli spaces, we may assume $a_i, b_i \leq c$ for all $i$, and only consider representations where maps along the top row of arrows are injective, maps along the bottom row of arrows are surjective, and those along the loops are nilpotent.  Such a representation determines configuration data $((W_i), (V_j), (\phi_k), \kk^c)$ where:
\begin{itemize}
\item $\kk^c$ is regarded as the ambient space,
\item $(W_i) =(W_1, \dotsc, W_m)$ is the sequence of subspaces of $\kk^c$ given by the images of the arrows along the top row,
\item $(V_j) =(V_1, \dotsc, V_n)$ is the sequence of subspaces of $\kk^c$ given by the kernels of the arrows along the bottom row,
\item $(\phi_k)=(\phi_1, \dotsc, \phi_l)$ is the the sequence of nilpotent endormorphisms of $\kk^c$ associated to the loops.
\end{itemize}
In the moduli space $\cM(\bd)^{ss}_{\theta}$, certain unstable configurations are omitted (depending on $\theta$), and two configurations $((W_i), (V_j), (\phi_k), \kk^c)$ and $((W'_i), (V'_j), (\phi'_k), \kk^c)$ represent the same point if there is an ambient linear transformation $g \in GL(\kk^c)$ such that $g\cdot W_i = W'_i$, $g\cdot V_i= V'_i$, and $g\phi_ig^{-1}=\phi'_i$ for all $i$.  (Note that the last statement is not ``if and only if'', because strictly semistable representations can be identified in the moduli space.)

If we study moduli spaces for quotients $A=\kk Q/I$, each relation of the form $a_i \to c \to b_j$ in $I$ imposes the incidence condition that $W_i \subset V_j$, and relations involving the loops similarly specify that the images and kernels of these transformations contain certain subspaces of the configuration.
Allowing arbitrary relations (even without loops) leads to moduli spaces where any singularity type of finite type over $\ZZ$ can appear, for example by taking $c=3$ and $a_i=b_i = 1$ for all $i$ and considering  Mn\"{e}v's Universality Theorem \cite{mnev}.
Thus we restrict the kinds of relations we consider in order to be able to apply our results.

Let $Q' \subset Q$ be a subquiver containing $x$.  Let  $I'=\kk Q'_{\geq 2}$, and let $I \subset \kk Q$ be generated by $I'$ and a nilpotency relation for each loop not in $Q'$, and set $A=\kk Q/I$ and $A'=\kk Q'/I'$.
Now let $\bd$ be a dimension vector for $Q$, let $\bd'$ be its restriction to $Q'$, and let $C \subset \rep_A(\bd)$ an irreducible component.  Then we have that $\rep_A(\bd)$ is the product of $\rep_{A'}(\bd')$ with an affine space and additional factors which are nilpotent cones in matrix spaces.  So $C$ is the product of some irreducible component $C'\subset \rep_{A'}(\bd')$ with that same affine space and products of nilpotent cones in matrix spaces.

By Corollary \ref{cor:radsquarezero}, we know $C'$ is a normal variety, so $C$ is as well.  
Thus $\cM(C)^{ss}_\theta$ is normal by applying \cite[Prop. 2.3.11]{DKbook} to the definition of $\cM(C)^{ss}_\theta$, and we can apply Theorem \ref{thm:CKmain} to decompose the moduli space $\cM(C)^{ss}_\theta$.  We note that there is not necessarily any relation between $\cM(C)^{ss}_\theta$ and $\cM(C')^{ss}_\theta$; although $C$ and $C'$ have essentially the same singularities, their $GL(\bd)$-orbit structures can be drastically different.
\end{example}

\textbf{Competing interests declaration.} Competing interests: The author(s) declare none.

\bibliographystyle{alpha}
\bibliography{radsquare}

\end{document}